\documentclass{article}
\usepackage{graphicx} 

\title{Almost cohomology groups for Lie rings of finite dimension}
\author{Moreno Invitti }
\date{}
\usepackage{amsfonts}
\usepackage{refcheck}
\usepackage[utf8]{inputenc}
\usepackage{xcolor}
 
\usepackage{amssymb}
\usepackage{tikz-cd}
 \usepackage{amscd,amsmath}
\usepackage{amssymb}
\usepackage{amsthm}
\def\Ind#1#2{#1\setbox0=\hbox{$#1x$}\kern\wd0\hbox to 0pt{\hss$#1\mid$\hss}
\lower.9\ht0\hbox to 0pt{\hss$#1\smile$\hss}\kern\wd0}

\def\notind#1#2{#1\setbox0=\hbox{$#1x$}\kern\wd0
\hbox to 0pt{\mathchardef\nn=12854\hss$#1\nn$\kern1.4\wd0\hss}
\hbox to 0pt{\hss$#1\mid$\hss}\lower.9\ht0 \hbox to 0pt{\hss$#1\smile$\hss}\kern\wd0}

\usepackage{amsmath} 
\usepackage{amssymb}
\usepackage{tikz-cd}
 \usepackage{amscd,amsmath}
\usepackage{subfigure}
\usepackage{csquotes} 
\newtheorem{theorem}{Theorem}[section]
\newtheorem{corollary}[theorem]{Corollary}

\newtheorem{lemma}[theorem]{Lemma}

\newtheorem*{thmB,2}{Theorem B, Second version}
\DeclareMathOperator{\Dom}{Dom}
\usepackage{booktabs}
\usepackage{caption}
\usepackage{amssymb}
\usepackage{tikz-cd}
\usepackage{mathptmx} 
\usepackage{graphicx} 
\usepackage{amsthm}
\usepackage{amsfonts}
\usepackage[utf8]{inputenc}
\usepackage{xcolor}

\usepackage{amssymb}
\usepackage{tikz-cd}
 \usepackage{amscd,amsmath}
\usepackage{amssymb}
\usepackage{amsthm}
\def\Ind#1#2{#1\setbox0=\hbox{$#1x$}\kern\wd0\hbox to 0pt{\hss$#1\mid$\hss}
\lower.9\ht0\hbox to 0pt{\hss$#1\smile$\hss}\kern\wd0}

\def\notind#1#2{#1\setbox0=\hbox{$#1x$}\kern\wd0
\hbox to 0pt{\mathchardef\nn=12854\hss$#1\nn$\kern1.4\wd0\hss}
\hbox to 0pt{\hss$#1\mid$\hss}\lower.9\ht0 \hbox to 0pt{\hss$#1\smile$\hss}\kern\wd0}

\usepackage{amsmath} 
\usepackage{amssymb}
\usepackage{wasysym}
\usepackage{tikz-cd}
\usepackage{amsmath}
\usepackage{subfigure}

\DeclareMathOperator{\ad}{ad}
\newtheorem{definition}[theorem]{Definition}

\newtheorem{question}{Question}

\usepackage{amssymb}
\usepackage{tikz-cd}
\usepackage{mathptmx} 
\usepackage{graphicx} 
\usepackage{refcheck}
\usepackage{hyperref}
\title{Almost cohomology of finite-dimensional Lie rings}
\author{Moreno Invitti}
\begin{document}
\maketitle
\begin{abstract}
    We introduce almost cohomology groups for Lie rings definable in finite-dimensional theory. In particular, we define the $0$th and $1$st almost cohomology groups of a Lie ring module. Moreover, we prove that the $1$st almost cohomology group of a finite-dimensional definable Lie ring module is finite if the $0$th almost cohomology group is finite.
\end{abstract}
\section{Introduction}
Cohomological tools play a fundamental role in the analysis of several algebraic structures, such as groups \cite{hochschild1953cohomologygroups} and Lie algebras \cite{hochschild1953cohomology}. In particular, they have been successfully applied in the classification of group extensions \cite{serre2013local}. For Lie algebras, several structural results concerning soluble Lie algebras have been obtained by Barnes \cite{barnes1967cohomology}.\\
A first model-theoretic approach to cohomology was pursued by Zamour in \cite{zamourCohom}, who proved that the $1$st cohomology group of a Lie ring definable in a finite-dimensional$^{\circ}$ (which are finite-dimensional theories that satisfy the DCC) is trivial if the $0$th cohomology group is trivial.\\
\emph{Finite-dimensional theories}, introduced by Wagner \cite{wagner2020dimensional}, represent an interesting extension of both supersimple theories of finite Lascar rank, and $o$-minimal theories.
\begin{definition}
    A theory $T$ is a \emph{finite-dimensional theory} if there exists a function $\operatorname{dim}$ from the class of all interpretable subsets in any model $\mathcal{M}$ of $T$ into $\omega\cup\{-\infty\}$ such that, for every formula $\phi(x,y)$, interpretable sets $X,Y$ in $T$ and interpretable function $f:X\to Y$, the following hold:
\begin{itemize}
    \item (\emph{Invariance}) If $a$ and $a'$ have the same type over $\emptyset$, then $\operatorname{dim}(\phi(x,a))=\operatorname{dim}(\phi(x,a'))$; 
    \item (\emph{Algebraicity}) $\operatorname{dim}(\emptyset)=-\infty$, and $\operatorname{dim}(X)=0$ if and only if $X$ is finite;
    \item (\emph{Union}) $\operatorname{dim}(X\cup Y)=\max\{\operatorname{dim}(X),\operatorname{dim}(Y)\}$;
    \item (\emph{Upper Fibration}) If $\operatorname{dim}(f^{-1}(y))\geq k$ for every $y\in Y$, then $\operatorname{dim}(X)\geq \operatorname{dim}(Y)+k$;
    \item (\emph{Lower Fibration}) If $\operatorname{dim}(f^{-1}(y))\leq k$ for every $y\in Y$, then $\operatorname{dim}(X)\leq \operatorname{dim}(Y)+k$.
\end{itemize}
\end{definition}
\begin{definition}
    Let $T$ be a finite-dimensional theory. We say that $T$ is a \emph{finite-dimensional$^{\circ}$} theory if every group $G$ definable in a model of $T$ satisfies the DCC, i.e., there are no infinite strictly descending chains of definable subgroups of $G$.
\end{definition}
\begin{definition}
    Let $\mathfrak{M}$ be a first-order structure. $\mathfrak{M}$ is said to be a \emph{finite-dimensional structure}, or equivalently a \emph{structure of finite dimension}, if $T(\mathfrak{M})$ is a finite-dimensional theory.
\end{definition}
Groups definable in $o$-minimal theories and those of finite Morley rank satisfy the DCC \cite{poizat,berarducci2005descending}. Nevertheless, there exist several superstable and supersimple groups, such as $(\mathbb{Z},+,0)$, that do not satisfy the DCC. Therefore, the analysis of non-connected Lie rings of finite dimension is a key point for the classification of "tame" Lie rings. This analysis was pursued in \cite{invitti2025lie}, which represents the basis of this article. In this framework, the usual cohomology groups do not encode the right information, as we will see in Section \ref{AlmCohLieRin}. Consequently, it is natural to introduce an extension of the Lie ring cohomology defined in \cite{zamourCohom}, called \emph{almost cohomology}. In particular, we define the $0$th and $1$st \emph{almost cohomology group} of a module $(\mathfrak{g},A)$, denoted $\widetilde{H}^0(\mathfrak{g},A)$ and $\widetilde{H}^1(\mathfrak{g},A)$ respectively. It remains open how to define higher cohomology groups in general, and their role remains unclear in the development of the study of finite-dimensional Lie rings (see Section \ref{QuestionLieRing}). Nevertheless, these two groups already suffice to derive the following theorem, which has several applications.
\begin{theorem}\label{TrivialityH1}
    Let $(\mathfrak{g},A)$ be a module definable in a finite-dimensional theory. Assume that $\mathfrak{g}$ is nilpotent and that $\widetilde{H}^0(\mathfrak{g},A)$ is finite. Then $\widetilde{H}^1(\mathfrak{g},A)$ is finite. 
\end{theorem}
Theorem \ref{TrivialityH1} is essential in the proof of the following corollary concerning \emph{almost Cartan subrings} (see Definition \ref{Def:AlmCarLieRin}).
\begin{corollary}
    Let $\mathfrak{g}$ be a Lie ring definable in a finite-dimensional theory. Let $\mathfrak{i}$ be a definable ideal of $\mathfrak{g}$ and $\mathfrak{c}$ a definable almost Cartan subring of $\mathfrak{i}$. Then $\mathfrak{g}\sim \mathfrak{i}+\widetilde{N}_{\mathfrak{g}}(\mathfrak{c})$.
\end{corollary}
\section{Preliminaries}
In this section, we introduce the tools necessary for this article.
\subsection{Group theory}\label{AlmGroThe}
We first recall the notions of \emph{almost containment} and \emph{commensurability}.
\begin{definition}\label{Def:AlmContain}
    Let $G$ be a group, and let $H,K$ be two subgroups of $G$. 
    \begin{itemize}
        \item The subgroup $H$ is \emph{almost contained} in $K$, denoted by $H\apprle K$, if $H\cap K$ has finite index in $H$.
        \item  The subgroup $H$ is \emph{commensurable} to $K$, denoted $H\sim K$, if $H\apprle K$ and $K\apprle H$.
        \item  A family $\{H_i\}_{i\in I}$ of subgroup of $G$ is \emph{uniformly commensurable} if there exists some $n\in \mathbb{N}$ such that $|H_i:H_i\cap H_j|\leq n$ for every $i,j\in I$.
    \end{itemize}
\end{definition}
A related notion is that of \emph{isogenies}.
\begin{definition}
    Let $G,H$ be two groups. Then $G$ and $H$ are \emph{isogenous} if there exists an isomorphism
    $$\phi:A/A_1\to B/B_1$$
    with $A,B$ subgroups of finite index in $G,H$ respectively and $A_1,B_1$ finite subgroups of $A,B$ respectively. The map $\phi$ is said to be an \emph{isogeny} between $G$ and $H$.
\end{definition}
All isogenies considered in this article are definable.\\
We recall some of the most important chain conditions on definable groups, which play a fundamental role in this article.
\begin{definition}
Let $G$ be a definable group. Then
\begin{itemize}
    \item The group $G$ satisfies the \emph{uniform descending chain condition}, or UCC, if $G$ does not admit an infinite strictly descending chain of uniformly definable subgroups.
    \item The group $G$ satisfies the \emph{infinite descending chain condition on definable subgroups}, or $\omega$-DCC, if $G$ does not admit an infinite strictly descending chain of definable subgroups $\{G_i\}_{i<\omega}$ such that $|G_i: G_{i+1}|\geq \omega$.
\end{itemize}
\end{definition}
\subsection{Finite-dimensional groups}
We now present some key results of finite-dimensional groups. 
We first show that a finite-dimensional group satisfies the UCC. This result is based on the following theorem (\cite[Theorem 2.5]{invitti2025lie}), which plays a fundamental role also in the proof that finite-dimensional groups are \emph{hereditarily $\widetilde{\mathfrak{M}}_c$}.
\begin{definition}
    A group $G$ is \emph{hereditarily $\widetilde{\mathfrak{M}}_c$} if, for every pair of definable subgroups $H,N$ such that $N$ is normalised by $H$, there exist natural numbers $n_{HN}$ and $d_{HN}$ such that every sequence of centralisers 
    $$C_H(a_0/N)\geq C_H(a_0,a_1/N)\geq\ldots\geq C_H(a_0,a_1,\ldots,a_n/N)\geq\ldots$$
    with $|C_H(a_0,\ldots,a_m/N):C_H(a_0,\ldots,a_{m+1}/N)|\geq d_{HN}$ for every $m<n$, has length at most $n_{HN}$.
\end{definition}
\begin{theorem}\label{ThmB}
    Let $G$ be a definable finite-dimensional group. Then, for every family of uniformly definable subgroups $\{H_i\}_{i\in I}$, there exist $n,k\in \mathbb{N}$ such that there are no subfamilies $\{H_{j}\}_{j=1}^N$ with $N>n$ such that $|\bigcap_{j\leq m} H_j/\bigcap_{j\leq m+1} H_{j}|\geq k$ for every $m\geq 1$.
\end{theorem}
If we apply these results to the centralisers of elements in $G$, we obtain that a finite-dimensional definable group has the $\widetilde{\mathfrak{M}}_c$-property.
\begin{lemma}
    Let $G$ be a finite-dimensional group. Then $G$ is a $\widetilde{\mathfrak{M}}_c$-group.
\end{lemma}
Another fundamental consequence of Theorem \ref{ThmB} is that finite-dimensional groups satisfy the UCC.
\begin{lemma}\label{ucc}
    Let $G$ be a finite-dimensional group. Then $G$ satisfies the UCC.
\end{lemma}
\begin{proof}
    Assume otherwise. Let $\{G_i\}_{i\in I}$ be a strictly descending chain of uniformly definable subgroups of $G$. Let $n\in \mathbb{N}$ be such that $\dim(G_n)$ is minimal. Then $\{G_{n+i}\}_{i\in \mathbb{N}}$ is a chain of uniformly definable subgroups of finite index in $G_n$. Since this chain is infinite and strictly descending, we have that for each $m\in \mathbb{N}$, there exists $M\in \mathbb{N}$ such that 
    $$|G_n\cap G_{n+M}|\geq m.$$
    This contradicts Theorem \ref{ThmB}.
\end{proof}
An important property of finite-dimensional fields is their perfection.
\begin{lemma}\label{perfection}
   Let $K$ be a field definable in a finite-dimensional theory. Then $K$ is perfect.
\end{lemma}
\begin{proof}
   In characteristic $0$, we have nothing to prove. Assume now that $\operatorname{char}(K)=p\neq 0$. Then $K^p$ is a definable subfield of $K$. By \cite[Proposition 3.3]{wagner2020dimensional}, we have that $\dim(K)=\dim(K^p)\operatorname{lin.dim}_{K^p}(K)$, where $\operatorname{lin.dim}_{K^p}(K)$ denotes the dimension as $K^p$-vector space of $K$. The Frobenius homomorphism $\varphi:K\to K$ with $\varphi(k)=k^p$ has finite kernel, and the image equals $K^p$. It follows from the fibration property that $\dim(K)=\dim(K^p)$. Hence $\operatorname{lin.dim}_{K^p}(K)=1$, and $K=K^p$.
\end{proof}
A consequence of Lemma \ref{perfection} is that a field of finite dimension does not admit a definable \emph{derivation}. This result is fundamental to prove Theorem \ref{TrivialityH1}.
\begin{definition}
    Let $K$ be a field. An additive homomorphism $\sigma:K\to K$ is a \emph{derivation} of $K$ if, for every $k,k'\in K$,
    $$\sigma(kk')=k\sigma(k')+k'\sigma(k).$$
\end{definition}
\begin{lemma}\label{DerTri}
    A finite-dimensional field $K$ admits no non-trivial definable derivations.
\end{lemma}
\begin{proof}
Let $\delta$ be a definable derivation of $K$.\\
    Assume first that $\operatorname{char}(K)=p$. By Lemma \ref{perfection}, for every $x\in K$, there exists some $y\in K$ such that $y^p=x$. Therefore $\delta(x)=\delta(y^p)=py^{p-1}\delta(y)=0$. Consequently, a finite-dimensional field of prime characteristic has no non-trivial derivation (even non-definable).\\
    Assume now that $\operatorname{char}(K)=0$. We show that the field of constant $C:=\{k\in K:\ \delta(k)=0\}$ of $\delta$ is definable and relatively algebraically closed. Hence $C$ is an infinite algebraically closed subfield of a finite-dimensional field, contradicting \cite[Proposition 3.3]{wagner2020dimensional}. It suffices to prove that $y\in C$ for every element $y\in K$ algebraic over $C$. Let $f(x)=\sum_{i=0}^n a_ix^i$ with $a_i\in C$ be the minimal polynomial of $y$ over $C$. Since $f(y)=0$, 
    $$0=\delta(0)=\delta(f(y))=\bigg(\sum_{i=1}^n a_i\cdot i\cdot y^{i-1}\bigg)\delta(y).$$ 
    Let $g(x)=\sum_{i=1}^n a_i\cdot i\cdot x^{i-1}$. Then $g(x)$ is a polynomial of degree strictly less than $\mathrm{deg}(f(x))$. By the minimality of $\mathrm{deg}(f(x))$, we have that $g(y)\neq 0$. Thus $g(y)\delta(y)=0$, and hence $y\in C$. In conclusion, the subfield $C$ is relatively algebraically closed, and the proof is completed.
\end{proof}
\subsection{Lie rings}
In this subsection, we define Lie rings and introduce their basic properties.
\begin{definition}
    A \emph{Lie ring} $(\mathfrak{g},+,[\cdot,\cdot])$ is an abelian group $(\mathfrak{g},+)$ with a map 
    $$[\cdot,\cdot]:\mathfrak{g}\times \mathfrak{g}\to \mathfrak{g}$$
    such that 
    \begin{itemize}
        \item $[\cdot,\cdot]$ is bilinear;
        \item $[\cdot,\cdot]$ is antisymmetric;
        \item $[\cdot,\cdot]$ satisfies the \emph{Jacobi identity}: for every $g_1,g_2,g_3\in \mathfrak{g}$, we have that $[g_1,[g_2,g_3]]+[g_3,[g_1,g_2]]+[g_2,[g_3,g_1]]=0$. 
    \end{itemize}
    The map $[\cdot,\cdot]$ is called the \emph{(Lie) bracket} on $\mathfrak{g}$.
\end{definition}
A basic example of a Lie ring is the group of endomorphisms of an abelian group $A$ equipped with the bracket $[f,g]:=f\circ g-g\circ f$, where $\circ$ denotes the composition in $\operatorname{End}(A)$.\\
The definitions of \emph{Lie subring} and \emph{ideal} for Lie rings resemble those for Lie algebras.
\begin{definition}
    A subgroup $\mathfrak{c}\leq \mathfrak{g}$ is a \emph{Lie subring}, denoted $\mathfrak{c}\sqsubseteq \mathfrak{g}$, if $\mathfrak{c}$ is closed under $[\cdot,\cdot]$. A Lie subring $\mathfrak{c}$ is an \emph{ideal} in $\mathfrak{g}$, denoted $\mathfrak{c}\mathrel{\unlhd}\mathfrak{g}$, if $[g,\mathfrak{c}]\leq \mathfrak{c}$ for all $g\in \mathfrak{g}$. 
\end{definition}
\emph{Simplicity} is a central concept in the theory of Lie rings. We also introduce a definable version of simplicity, namely \emph{definable simplicity}.
\begin{definition}
A Lie ring $\mathfrak{g}$ is \emph{simple} if $\mathfrak{g}$ has no non-trivial ideal, i.e., distinct from $\{0\}$ and $\mathfrak{g}$. A definable Lie ring $\mathfrak{g}$ is \emph{definably simple} if it has no definable proper ideals.
\end{definition}
Basic examples of ideals include the $\mathfrak{g}[n]=\{g\in \mathfrak{g}:\ ng=0\}\mathrel{\unlhd} \mathfrak{g}$ for every $n\in \mathbb{N}$. Therefore, for every simple Lie ring $\mathfrak{g}$, the group $(\mathfrak{g},+)$ is either of prime exponent or it is torsion-free. In the first case, we say that $\mathfrak{g}$ has \emph{characteristic $p$}, in the latter that $\mathfrak{g}$ has \emph{characteristic $0$}.
\begin{definition}\label{DefinChar}
    Let $\mathfrak{g}$ be a Lie ring. If $\mathfrak{g}=\mathfrak{g}[p]$, we say that $\mathfrak{g}$ has \emph{characteristic $p$}. If $(\mathfrak{g},+)$ is torsion-free, we say that $\mathfrak{g}$ has \emph{characteristic $0$}. The characteristic of a Lie ring $\mathfrak{g}$ is denoted by $\operatorname{char}(\mathfrak{g})$.
\end{definition}
We define the notion of a \emph{Lie ring homomorphism}.
\begin{definition}
    Let $\mathfrak{g}$ and $\mathfrak{h}$ be two Lie rings. A map $\phi:\mathfrak{g}\to \mathfrak{h}$ is a \emph{Lie ring homomorphism} if $\phi$ is a homomorphism of abelian groups from $(\mathfrak{g},+)$ to $(\mathfrak{h},+)$ such that $\phi([g,g'])=[\phi(g),\phi(g')]$ for every $g,g'\in \mathfrak{g}$.
\end{definition}
We now focus on the \emph{action of a Lie ring} on an abelian group.
\begin{definition}
 Let $\mathfrak{g}$ be a Lie ring. A \emph{$\mathfrak{g}$-module}, or simply a module, is a triple $(\mathfrak{g},V,\cdot)$ with $V$ an abelian group and $\cdot:\mathfrak{g}\times V\to V$ a map that induces a homomorphism of Lie rings $\phi:\mathfrak{g}\to \operatorname{End}(V)$ given by $\phi(g)(v)=g\cdot v$. The map $\phi$ is called the \emph{representation} of $\mathfrak{g}$ over $V$. Equivalently, we say that $\mathfrak{g}$ \emph{acts} on $V$. We usually omit $\cdot$.
\end{definition}
A basic example of Lie ring action is the \emph{adjoint action} of $\mathfrak{g}$ on itself.
\begin{definition}
    Let $\mathfrak{g}$ be a Lie ring. Then the action $\operatorname{ad}:\mathfrak{g}\times\mathfrak{g}\to \mathfrak{g}$ that sends $(g,g_1)\in \mathfrak{g}\times \mathfrak{g}$ to $[g,g_1]\in \mathfrak{g}$ is called the \emph{adjoint action} of $\mathfrak{g}$ on itself.
\end{definition}
In this article, all the modules considered are \emph{definable}: the Lie ring $\mathfrak{g}$, the group $V$, and the map $\cdot$ are all definable. Two fundamental classes of modules are \emph{faithful} and \emph{irreducible (or minimal)} modules.
\begin{definition}
Let $\mathfrak{g}$ be a definable Lie ring, and let $(\mathfrak{g},V)$ be a $\mathfrak{g}$-module with $\phi$ the associated representation. Then $(\mathfrak{g},V)$ is said to be \emph{faithful} if $\ker(\phi)$ is trivial.
\end{definition}
\begin{definition}
Let $V$ be a definable group, and $X\subseteq \operatorname{End}(V)$. A definable subgroup $W\leq V$ is a \emph{$X$-invariant} subgroup, or equivalently $W$ is \emph{$X$-invariant}, if for every $x\in X$, we have that $x(W)\leq W$. The group $V$ is \emph{$X$-minimal} if $V$ has no infinite definable $X$-invariant subgroups of infinite index.\\
Let $\mathfrak{g}$-module $(\mathfrak{g},V)$ is \emph{irreducible}, or equivalently \emph{minimal}, if $V$ is $\mathfrak{g}$-minimal (with $\mathfrak{g}$ seen as a subgroup of $\operatorname{End}(V)$).
\end{definition}
Central tools in the analysis of Lie ring modules are the \emph{centralisers} of a point.
\begin{definition}
    Let $\mathfrak{g}$ be a Lie ring, $V$ a $\mathfrak{g}$-module and $A\leq V$ a subgroup. Let $v\in V$, then the \emph{centraliser of $v$ in $\mathfrak{g}$ over $A$}, denoted $C_{\mathfrak{g}}(v/A)$, is the subgroup
    $$C_{\mathfrak{g}}(v/A):=\{g\in \mathfrak{g}: g\cdot v\in A\}.$$
    Let $g\in \mathfrak{g}$, the \emph{centraliser of $g$ in $V$ over $A$}, denoted $C_V(g/A)$, is the subgroup 
    $$C_V(g/A):=\{v\in V:\ gv\in A\}.$$
    If the subgroup $A=\{0\}$, it will be omitted. 
\end{definition}
We can also define the centralisers for a subgroup of $V$.
\begin{definition}
    Let $\mathfrak{g}$ be a Lie ring and $V$ a $\mathfrak{g}$-module. Let $A,B$ be subgroups of $V$ and $\mathfrak{h}$ a Lie subring of $\mathfrak{g}$. Define the \emph{centraliser of $B$ in $\mathfrak{g}$ over $A$} as the subgroup 
    $$C_{\mathfrak{g}}(B/A):=\bigcap_{b\in B} C_{\mathfrak{g}}(b/A).$$
    Define the \emph{centraliser of $\mathfrak{h}$ in $V$ over $A$} as the subgroup 
    $$C_V(\mathfrak{h}/A):=\bigcap_{h\in \mathfrak{h}}C_V(h/A).$$
\end{definition}
We now show that, if $A$ is a $\mathfrak{g}$-submodule of $V$, then $C_{\mathfrak{g}}(B/A)$ is an ideal for every $\mathfrak{g}$-submodule $B$. Moreover $C_V(\mathfrak{h}/A)$ is a $\mathfrak{g}$-submodule for every ideal $\mathfrak{h}$. An analogue result for \emph{almost centralisers} is proved in Subsection \ref{AlmostLiering} (Lemma \ref{DefZ}).
\begin{lemma}\label{IdeCenAct}
    Let $\mathfrak{g}$ be a Lie ring. Let $V$ be a $\mathfrak{g}$-module and $A\leq V$ a $\mathfrak{g}$-submodule.
    \begin{enumerate}
        \item If $B\leq V$ is a $\mathfrak{g}$-submodule, then $C_{\mathfrak{g}}(B/A)$ is an ideal in $\mathfrak{g}$;
        \item if $\mathfrak{h}\mathrel{\unlhd}\mathfrak{g}$ is an ideal of $\mathfrak{g}$, then $C_{V}(\mathfrak{h}/A)$ is a $\mathfrak{g}$-submodule. 
    \end{enumerate} 
\end{lemma}
\begin{proof}
    (1) It suffices to show that, for every $g\in \mathfrak{g}$ and $c\in C_{\mathfrak{g}}(B/A)$, the element $[g,c](b)$ belongs to $A$ for every $b\in B$. Since $\mathfrak{g}$ acts as a Lie ring on $V$,
    $$[g,c](b)=g(c(b))+c(g(b))\in g(A)+c(B)\leq A+A\leq A,$$
    as $A,B$ are $\mathfrak{g}$-submodules.\\
    (2) It suffices to check that, for every $g\in \mathfrak{g}$, $c\in C_V(\mathfrak{h}/A)$ and $h\in \mathfrak{h}$, $h(gc)=0$.
    By the Jacobi identity, 
    $$h(gc)=[h,g](c)+g(h(c))=0$$
    since $[h,g]\in \mathfrak{h}$.
\end{proof}
We can adapt the previous notions of centralisers to this case.
\begin{definition}
    Let $\mathfrak{g}$ be a Lie ring, $A\leq \mathfrak{g}$ a subgroup, and $g\in \mathfrak{g}$. Then the \emph{centraliser of $g$ over $A$} is the centraliser of $g$ in $\mathfrak{g}$ over $A$ for the adjoint action. We denote by $C_{\mathfrak{g}}(g/A)$ the centraliser of $g$ over $A$.
\end{definition}
Similarly, one can define the centraliser for a subgroup. Applying Lemma \ref{IdeCenAct} to the adjoint action, we obtain the following result.
\begin{lemma}\label{IdeCen}
    Let $\mathfrak{g}$ be a Lie ring. Let $H,K\mathrel{\unlhd}\mathfrak{g}$ be ideals. Then $C_{\mathfrak{g}}(H/K)$ is an ideal of $\mathfrak{g}$. 
\end{lemma}
We define \emph{nilpotency} for Lie rings.
\begin{definition}
Let $\mathfrak{g}$ be a Lie ring. 
\begin{itemize}
    \item The Lie ring $\mathfrak{g}$ is \emph{abelian} if $\mathfrak{g}=Z(\mathfrak{g})$, where $Z(\mathfrak{g}):=\{g\in \mathfrak{g}:\ \forall g'\ [g,g']=0\}$;
    \item The Lie ring $\mathfrak{g}$ is \emph{nilpotent} if there exists a chain
        $$\mathfrak{g}=\mathfrak{g}_0\mathrel{\unrhd} \mathfrak{g}_1\mathrel{\unrhd}\ldots\mathrel{\unrhd}\mathfrak{g}_n=0$$
        of ideals of $\mathfrak{g}$ such that $[\mathfrak{g},\mathfrak{g}_i]\leq \mathfrak{g}_{i+1}$ for each $i\leq n-1$;
    
\end{itemize}
\end{definition}
The nilpotency is linked to the \emph{lower central series} $\big(\mathfrak{g}^{[n]}\big)_{n<\omega}$.
\begin{definition}
    Let $\mathfrak{g}$ be a Lie ring. We define $\mathfrak{g}^{[n]}$ by induction as follows:
    \begin{itemize}
        \item $\mathfrak{g}^{[0]}=\mathfrak{g}$;
        \item $\mathfrak{g}^{[n+1]}=[\mathfrak{g},\mathfrak{g}^{[n]}]$.
    \end{itemize}
\end{definition}
The proof of the following lemma is immediate.
\begin{lemma}
    Let $\mathfrak{g}$ be a Lie ring. $\mathfrak{g}$ is nilpotent if and only if there exists some $n\in\ \mathbb{N}$ such that $\mathfrak{g}^{[n]}=0$.
\end{lemma}
We now introduce the \emph{iterated centralisers}.
\begin{definition}
    Let $\mathfrak{g}$ be a Lie ring, $X\subseteq \mathfrak{g}$, and $H\leq \mathfrak{g}$. We denote by $C_{\mathfrak{g}}(X/H):=\{g'\in \mathfrak{g}:\ [X,g']\in H\}$ the \emph{centraliser of $X$ in $\mathfrak{g}$ over $H$}. We define the \emph{iterated centraliser} of a subset $X\subseteq \mathfrak{g}$ as follows:
    \begin{itemize}
        \item $C^1_{\mathfrak{g}}(X)=C_{\mathfrak{g}}(X)=\{g'\in \mathfrak{g}:\ [X,g']=0\}$;
        \item $C^{i+1}_{\mathfrak{g}}(X)=C_{\mathfrak{g}}(X/C^{i}_{\mathfrak{g}}(X))$.
    \end{itemize}
\end{definition}
In general, the centraliser $C_{\mathfrak{g}}(X/H)$ is only a subgroup of $\mathfrak{g}$. If $H$ is an ideal, then $C_{\mathfrak{g}}(X/H)$ is a Lie subring of $\mathfrak{g}$. Indeed, for every $c,c'\in C_{\mathfrak{g}}(X/H)$, 
$$[[c,c'],X]=[[X,c],c']+[[c',X],c]\leq [H,c']+[H,c]\leq H+H\leq H.$$
Typical examples of iterated centralisers, which are linked to the nilpotency of $\mathfrak{g}$, are the \emph{iterated centers of $\mathfrak{g}$}.
\begin{definition}
    We define recursively the \emph{$n$-center of $\mathfrak{g}$}, denoted $Z_n(\mathfrak{g})$, as follows:
    \begin{itemize}
        \item $Z_1(\mathfrak{g})=Z(\mathfrak{g})=C_{\mathfrak{g}}(\mathfrak{g})\mathrel{\unlhd}\mathfrak{g}$;
        \item $Z_{n+1}(\mathfrak{g})\mathrel{\unlhd} \mathfrak{g}$ such that $Z_{n+1}(\mathfrak{g})/Z_n(\mathfrak{g})=C_{\mathfrak{g}}(\mathfrak{g}/Z_n(\mathfrak{g}))$.
    \end{itemize}
\end{definition}
The following lemmas are direct applications of the definitions.
\begin{lemma}
    Let $\mathfrak{g}$ be a Lie ring. Then $Z_i(\mathfrak{g})$ is an ideal in $\mathfrak{g}$ for every $i\in \mathbb{N}$. Moreover, the Lie ring $\mathfrak{g}$ is nilpotent if and only if there exists some $n\in \mathbb{N}$ such that $Z_n(\mathfrak{g})=\mathfrak{g}$.
\end{lemma}
\begin{lemma}\label{FinCenConn}
    Let $\mathfrak{g}$ be a connected definable Lie ring, and assume that $Z(\mathfrak{g})$ is finite. Then $Z(\mathfrak{g}/Z(\mathfrak{g}))$ is trivial.
\end{lemma}
We now introduce the notion of \emph{normaliser} of a Lie subring.
\begin{definition}
    Let $\mathfrak{g}$ be a Lie ring and $\mathfrak{h}\sqsubseteq \mathfrak{g}$ a Lie subring. The \emph{normaliser of $\mathfrak{h}$ in $\mathfrak{g}$}, denoted $N_{\mathfrak{g}}(\mathfrak{h})$, is the Lie subring 
    $$N_{\mathfrak{g}}(\mathfrak{h})=\{g\in \mathfrak{g}:\ [g,\mathfrak{h}]\leq \mathfrak{h}\}.$$
\end{definition}
Cartan Lie algebras were introduced by \'Elie Cartan in his doctoral thesis and, since then, they have been crucial in the study of Lie algebras. Cartan Lie subrings naturally extend these objects to the setting of Lie rings.
\begin{definition}
    Let $\mathfrak{g}$ be a Lie ring and $\mathfrak{h}\sqsubseteq \mathfrak{g}$ a Lie subring. $\mathfrak{h}$ is said to be a \emph{Cartan Lie subring} if $\mathfrak{h}$ is nilpotent and \emph{self-normalising}, i.e., $N_{\mathfrak{g}}(\mathfrak{h})=\mathfrak{h}$.
\end{definition}
\subsection{Almost Lie ring theory}\label{AlmostLiering}
Section 4 of \cite{invitti2025lie} introduces the natural generalisations of the previous notions, and it establishes some fundamental properties of \emph{hereditarily $\widetilde{\mathfrak{M}}_c$-Lie rings}.
\begin{definition}
    Let $\mathfrak{g}$ be a definable Lie ring. Then $\mathfrak{g}$ is \emph{hereditarily $\widetilde{\mathfrak{M}}_c$} if, for any definable subgroup $N$ in $\mathfrak{g}$, there exist $n,d<\omega$ such that there cannot exist $g_1,\ldots,g_n\in \mathfrak{g}$ with $|C_{\mathfrak{g}}(g_1,\ldots,g_i/N):C_{\mathfrak{g}}(g_1,\ldots,g_i,g_{i+1}/N)|\geq d$ for every $i\leq n-1$.
\end{definition}
It follows from Theorem \ref{ThmB} that the results of Section 4 of \cite{invitti2025lie} can be applied to the action of a finite-dimensional Lie ring.\\
We define the \emph{almost centraliser} of a module $(\mathfrak{g},V)$.
\begin{definition}
    Let $\mathfrak{g}$ be a Lie ring and $(\mathfrak{g},V)$ a $\mathfrak{g}$-module. Let $A,B\leq V$ and $H\leq \mathfrak{g}$. We define:
\begin{itemize}
    \item the \emph{almost centraliser of $B$ in $\mathfrak{g}$ over $A$}, denoted $\widetilde{C}_{\mathfrak{g}}(B/A)$, as 
    $$\widetilde{C}_{\mathfrak{g}}(B/A)=\{g\in \mathfrak{g}:\ g\cdot B\apprle A\}.$$
    \item the \emph{almost centraliser of $H$ in $V$ over $A$}, denoted $\widetilde{C}_V(H/A)$ as 
    $$\widetilde{C}_V(H/A)=\{v\in V:\ H\cdot v\apprle A\}.$$
\end{itemize}
\end{definition}
The following lemma follows easily from \cite[Lemma 4.10]{invitti2025lie}.
\begin{lemma}\label{g/Z(g)}
    Let $\mathfrak{g}$ be a finite-dimensional Lie ring acting on the definable abelian group $V$ of finite dimension. Assume that $\widetilde{C}_V(\mathfrak{g})$ is finite. Then $\widetilde{C}_{V/\widetilde{C}_V(\mathfrak{g})}(\mathfrak{g})$ is trivial.
\end{lemma}
Two important properties of the almost centraliser of the action of a finite-dimensional Lie ring are \emph{definability} (\cite[Lemma 4.3]{invitti2025lie}) and \emph{symmetry} (\cite[Lemma 4.7]{invitti2025lie}).
\begin{lemma}\label{DefZ}
    Let $(\mathfrak{g},V)$ be a module. For every $\mathfrak{g}$-invariant subgroup $W\leq V$ and ideal $\mathfrak{h}\mathrel{\unlhd}\mathfrak{g}$, the subgroup $\widetilde{C}_V(\mathfrak{h})$ is a $\mathfrak{g}$-submodule and $\widetilde{C}_{\mathfrak{g}}(W)$ is a $\mathfrak{g}$-ideal. Moreover, if $(\mathfrak{g},V)$ is a finite-dimensional module and both $W$ and $\mathfrak{h}$ are definable, then both $\widetilde{C}_{\mathfrak{g}}(W)$ and $\widetilde{C}_V(\mathfrak{h})$ are definable.
\end{lemma}
\begin{lemma}\label{symact}
    Let $\mathfrak{g}$ be a Lie ring acting on an abelian group $V$, both definable in a finite-dimensional theory. Let $H\subseteq \mathfrak{g}$ be a definable subgroup, and let $A,N\leq V$ be definable subgroups. Then $H\apprle \widetilde{C}_{\mathfrak{g}}(A/N)$ if and only if $A\apprle \widetilde{C}_V(H/N)$. 
\end{lemma}
A key generalisation of $\mathfrak{g}$-modules are \emph{almost invariant $\mathfrak{g}$-modules}.
\begin{definition}\label{Def:AlmInv}
Let $\mathfrak{g}$ be a Lie ring, $V$ a $\mathfrak{g}$-module and $W\leq V$ a subgroup. The Lie subring
$$\widetilde{Stab}_{\mathfrak{g}}(W)=\{g\in \mathfrak{g}:\ g W\apprle W\}$$
is called the \emph{almost stabiliser} of $W$. The subgroup $W$ is \emph{almost $\mathfrak{g}$-invariant} if $\widetilde{Stab}_{\mathfrak{g}}(W)=\mathfrak{g}$.
\end{definition}
If the action is represented by the adjoint action of a Lie ring on itself, the almost stabiliser is called \emph{almost normaliser}.
\begin{definition}
    Let $\mathfrak{g}$ be a Lie ring, and $H\leq \mathfrak{g}$. Then the almost stabiliser of $H$ for the adjoint action of $\mathfrak{g}$ is called the \emph{almost normaliser of $H$ in $\mathfrak{g}$}, and denoted by $\widetilde{N}_{\mathfrak{g}}(H)$.
\end{definition}
The notion of \emph{(strongly) absolutely minimal module} plays a central role in the analysis of finite-dimensional Lie rings.
\begin{definition}
    A definable module $(\mathfrak{g},V)$ is \emph{absolutely minimal} if there is no definable infinite almost $\mathfrak{g}$-invariant subgroup of infinite index in $V$. A definable module $(\mathfrak{g},V)$ is \emph{strongly absolutely minimal} if no definable infinite subgroup $W\leq V$ of infinite index in $V$ satisfies $\widetilde{Stab}_{\mathfrak{g}}(W)\sim \mathfrak{g}$.
\end{definition}
The following characterisation of strongly absolutely minimal modules (\cite[Lemma 5.6]{invitti2025lie}) is essential in the proof of Theorem \ref{TrivialityH1}.
\begin{lemma}\label{LinAlmAbe2}
    Let $\mathfrak{g}$ be a Lie ring acting on the abelian group $A$, both definable in a finite-dimensional theory. Assume that $A$ is minimal for every definable Lie subring of finite index in $\mathfrak{g}$. Then $A$ is strongly absolutely $\mathfrak{g}$-minimal.
\end{lemma}
A consequence of Lemma \ref{LinAlmAbe2} is the following result (\cite[Lemma 5.7]{invitti2025lie}).
\begin{lemma}\label{ObsStrAbsMin}
    Let $(\mathfrak{g},A)$ be a definable finite-dimensional module. Then there exists a strongly absolutely minimal module $(\mathfrak{h},B)$, where $\mathfrak{h}\sqsubseteq \mathfrak{g}$ is a definable Lie subring of finite index and $B\leq A$ is a definable $\mathfrak{h}$-submodule. 
\end{lemma}
Another important result concerning finite-dimensional modules is the following (\cite[Lemma 4.17]{invitti2025lie}).
\begin{lemma}\label{fin[]}
    Let $(\mathfrak{g},V)$ be a definable module in a finite-dimensional theory. Then $[\widetilde{C}_{\mathfrak{g}}(V),\widetilde{C}_{V}(\mathfrak{g})]$ is finite.
\end{lemma}
The action of a nilpotent Lie ring of finite dimension is characterised by the following theorem (\cite[Theorem 6.9]{invitti2025lie}).
\begin{theorem}\label{LinAlmNil}
    Let $\mathfrak{g}$ be a definable nilpotent Lie ring acting on a definable module $V$, both definable in a finite-dimensional theory. Assume that the action is not almost trivial and that $V$ is strongly absolutely $\mathfrak{g}$-minimal. Then there exist:
    \begin{itemize}
        \item A definable Lie subring $\mathfrak{h}\sqsubseteq \mathfrak{g}$ of finite index in $\mathfrak{g}$;
        \item A definable $\mathfrak{h}$-submodule $V_1\leq V$ of finite index in $V$;
        \item A finite $\mathfrak{h}$-submodule $V_2\leq V_1$
    \end{itemize}
    such that $\mathfrak{h}/\widetilde{C}_{\mathfrak{h}}(V)$ is abelian and acts $K$-scalarly on the $\mathfrak{h}$-module $V_1/V_2\simeq K^+$, where $K$ is a definable field of finite dimension.
\end{theorem}
Finally, we introduce \emph{almost Cartan Lie subrings}.
\begin{definition}\label{Def:AlmCarLieRin}
    Let $\mathfrak{g}$ be a Lie ring, and let $\mathfrak{h}$ be a Lie subring. Then $\mathfrak{h}$ is an \emph{almost Cartan Lie subring} if $\mathfrak{h}$ is virtually nilpotent and \emph{almost self-normalising}, i.e., $\mathfrak{h}$ has finite index in $\widetilde{N}_{\mathfrak{g}}(\mathfrak{h})$.
\end{definition}
\section{Almost cohomology of Lie rings}\label{AlmCohLieRin}
In this section, we define the \emph{almost cohomology groups} of a Lie ring, and we establish some basic properties.\\
We begin by recalling the usual definition of cohomology groups for a Lie ring \cite{palais1961cohomology}. This introduction is crucial since it represents the basis for our new notion and clarifies the limits of the classical approach in the finite-dimensional setting. The first step in the construction of cohomology groups for Lie rings (or an algebraic structure in general) is the introduction of a \emph{chain complex}: a sequence of groups $(C^i(\mathfrak{g},A))_{i<\omega}$ and of homomorphisms $d_i:C^i(\mathfrak{g},A)\to C^{i+1}(\mathfrak{g},A)$, called the \emph{differentials}, such that $d^{i+1}\circ d^i=0$.  
    \begin{definition}
        Let $\mathfrak{g}$ be a Lie ring and $(\mathfrak{g},A)$ a $\mathfrak{g}$-module. The groups $C^i(\mathfrak{g},A)$ are defined as follows:
        \begin{itemize}
            \item $C^0(\mathfrak{g},A)=A$;
            \item $C^{i}(\mathfrak{g},A)=\{f:\mathfrak{g}^i\to A: f\text{ is an alternating homomorphism}\}$. A homomorphism $f:\mathfrak{g}^i\to A$ is \emph{alternating} if $f(\overline{g})=0$ for every $\overline{g}\in \mathfrak{g}^i$ such that $\overline{g}_k=\overline{g}_j$ for some $j\neq k\leq n$. Here $\overline{g}_j$ denote the $j$-th element of the tuple $\overline{g}$. 
        \end{itemize}
        Define the \emph{differential} $d^i:C^i(\mathfrak{g},A)\to C^{i+1}(\mathfrak{g},A)$ as the map that sends $f\in C^i(\mathfrak{g},A)$ to the homomorphism $d^i(f)$ such that $d^i(f)(g_1,\ldots, g_{i+1})$ equals 
        $$\sum_{1\leq s<t\leq i+1} (-1)^{s+t-1}f([g_s,g_t],g_1,\ldots,\hat{g_s}\ldots \hat{g_t},\ldots,g_{i+1})+\sum_{1\leq j\leq i+1}(-1)^j g_jf(g_1,\ldots, \hat{g_j}\ldots g_{i+1}).$$
    \end{definition}
    The $i$th cohomology group of the module $(\mathfrak{g},A)$ is the group $H^i(\mathfrak{g},A):=\ker(d^i)/\mathrm{im}(d^{i-1})$ for each $i\in \mathbb{N}$. An important property of the $0$th cohomology group is that $H^0(\mathfrak{g},A)=C_{A}(\mathfrak{g})$. In finite-dimensional Lie rings, the central notion is represented by the almost centraliser of the action rather than the centraliser. Hence, it is crucial to introduce a cohomology theory in which the $0$th cohomology group coincides with $\widetilde{C}_A(\mathfrak{g})$.\\
    We now define the groups $\widetilde{C}^0(\mathfrak{g},A)$ and $\widetilde{C}^1(\mathfrak{g},A)$. Before doing so, we need to introduce the set of \emph{partial homomorphisms} from a group $G$ to another group $H$.
    \begin{definition}
        Let $G,H$ be two abelian groups. A \emph{partial homomorphism} is a homomorphism $f:G_1\to H$, where $G_1$ is a subgroup of finite index in $G$. The subgroup $G_1$ is called the \emph{domain} of $f$, and it is denoted by $\Dom(f)$. The set of partial homomorphisms from $G$ to $H$ is denoted by $\mathcal{P}\text{-}\operatorname{Hom}(G,H)$.
    \end{definition}
    If the groups $G$ and $H$ are abelian, we can equip $\mathcal{P}\text{-}\operatorname{Hom}(G,H)$ with an operation $+$.
    \begin{definition}
        Let $G,H$ be two abelian groups, and let $f,f_1\in \mathcal{P}\text{-}\operatorname{Hom}(G,H)$. Then we define the \emph{sum}
        $$f+f_1:\Dom(f)\cap \Dom(f_1)\to H$$
        as the partial homomorphism sending $g\in \Dom(f)\cap \Dom(f_1)$ to $(f+f_1)(g)=f(g)+f_1(g)$. The trivial homomorphism from $G$ to $H$ is the identity element for this operation. We define $-f:\Dom(f)\to H$ as the partial homomorphism sending $g\in \Dom(f)$ to $(-f)(g)=-(f(g))$.
    \end{definition}
    The structure $(\mathcal{P}\text{-}\operatorname{Hom}(G,H),+,0)$ is a monoid but not a group in general. Indeed $f+(-f)\neq 0$ since, if $\Dom(f)\neq G$, the domains of $0$ and $f+(-f)$ are different. We may introduce an equivalence relation $\sim$ such that $(\mathcal{P}\text{-}\operatorname{Hom}(G,H)/\sim,+,0)$ is a group.
    \begin{definition}
        Let $G,H$ be two abelian groups. We define the equivalence relation $\sim$ on $\mathcal{P}\text{-}\operatorname{Hom}(G,H)$ as follows: for every $f,f_1\in \mathcal{P}\text{-}\operatorname{Hom}(G,H)$, we have that $f\sim f_1$ if and only if $(f-f_1)(\Dom(f-f_1))$ is finite.
    \end{definition}
 A standard calculation shows that $(\mathcal{P}\text{-}\operatorname{Hom}(G,H)/\sim,+,0)$ is a group.
 \begin{lemma}\label{PHomGro}
     Let $G,H$ be two abelian groups. The structure $(\mathcal{P}\text{-}\operatorname{Hom}(G,H)/\sim,+,0)$ is a group.
 \end{lemma}
We now define $\widetilde{C}^0(\mathfrak{g},A)$ and $\widetilde{C}^1(\mathfrak{g},A)$ for a module $(\mathfrak{g},A)$.
    \begin{definition}
    Let $\mathfrak{g}$ be a Lie ring and $A$ a $\mathfrak{g}$-module. We define $\widetilde{C}^0(\mathfrak{g},A)$ and $\widetilde{C}^1(\mathfrak{g},A)$ as follows:
    \begin{itemize}
        \item $\widetilde{C}^0(\mathfrak{g},A)=A$;
        \item $\widetilde{C}^1(\mathfrak{g},A)=\mathcal{P}\text{-}\operatorname{Hom}(\mathfrak{g},A)/\sim$.
    \end{itemize}
    We define the differential $$d:\widetilde{C}^0(\mathfrak{g},A)\to \widetilde{C}^{1}(\mathfrak{g},A)$$
    as the map sending $a$ to $[d(a)]$ where $d(a):\mathfrak{g}\to A$ is the map that sends $g\in \mathfrak{g}$ to $d(a)(g)=g(a)$. 
    \end{definition}
    It follows from Lemma \ref{PHomGro} that $(\widetilde{C}^1(\mathfrak{g},A),+)$ is a group. Note that $\widetilde{C}^1(\mathfrak{g},A)$ is the natural extension of the groups $C^1(\mathfrak{g},A)$ in the finite-dimensional context, up to working with partial homomorphisms. The reason for the introduction of partial endomorphisms will be clarified in Lemma \ref{ExactserC}. The main difficulty in the definition of $\widetilde{C}^2(\mathfrak{g},A)$ lies in the construction of the differential map (see Section \ref{QuestionLieRing}). 
    We associate with $\widetilde{C}^0(\mathfrak{g},A)$ and $\widetilde{C}^1(\mathfrak{g},A)$ the \emph{almost cohomology groups} $\widetilde{H}^{0}(\mathfrak{g},A)=\ker(d)$, called the \emph{$0$th almost cohomology group}, and $\widetilde{H}^1=\operatorname{AlmDer}(\mathfrak{g},A)/\mathrm{im}(d)$, called the \emph{$1$st almost cohomology group}, where $\operatorname{AlmDer}(\mathfrak{g},A)$ is the group of \emph{almost derivations} of the module $(\mathfrak{g},A)$. 
    \begin{definition}\label{DefAlmDer}
    Let $\mathfrak{g}$ be a Lie ring and $A$ a $\mathfrak{g}$-module. Let $f:\Dom(f)\to A$ be an element of $\mathcal{P}\text{-}\operatorname{Hom}(\mathfrak{g},A)$. Define, for $g\in \Dom(f)$, the subgroups
    $$D^f_{g}:=\{g'\in \Dom(f)\cap \ad_g^{-1}(\Dom(f)):\ f([g,g'])=g(f(g'))-g'(f(g))\}$$
    and 
    $$D^f:=\{g\in \Dom(f):\ D^f_g\sim \mathfrak{g}\}.$$
    Then $f$ is said to be an \emph{almost derivation} if $D^f$ has finite index in $\mathfrak{g}$.
    \end{definition}
    We now show that $\operatorname{AlmDer}(\mathfrak{g},A)$ is a subgroup of $\widetilde{C}^1(\mathfrak{g},A)$.
    \begin{lemma}\label{AlmDer}
        Let $\mathfrak{g}$ be a Lie ring and $(\mathfrak{g},A)$ a $\mathfrak{g}$-module. Then, for every $f,h\in \mathcal{P}\text{-}\operatorname{Hom}(\mathfrak{g},A)$ and $g\in \mathfrak{g}$, we have that
        \begin{enumerate}
            \item $D^f_g$ is a subgroup of $\mathfrak{g}$;
            \item $D^f$ is a subgroup of $\mathfrak{g}$;
            \item If $f,h$ are almost derivations, then $f+h$ and $-f$ are almost derivation;
            \item If $f$ is an almost derivation and $h\sim f$, then $h$ is an almost derivation.
        \end{enumerate}
    \end{lemma}
    \begin{proof}
        (1) Let $g_1,g_2\in D^f_g$. Then, by the linearity of $f$ and of the Lie bracket,
        \begin{align*}
            f([g,g_1+g_2])&=f([g,g_1])+f([g,g_2])=gf(g_1)-g_1f(g)+gf(g_2)-g_2f(g)\\
            &=gf(g_1+g_2)-(g_1+g_2)f(g).
        \end{align*}
        The closure under the additive inverse follows similarly.\\
        (2) Let $g,g'\in \mathfrak{g}$ be such that $D^f_g\sim D_{g'}^f\sim \mathfrak{g}$. Then $D^f_g\cap D^f_{g'}\leq \mathfrak{g}$ is a subgroup finite index. Moreover, for every $g_1\in D^f_g\cap D^f_{g'}$, we have that $[g+g',g_1]=[g,g_1]+[g',g_1]\in \Dom(f)$ since $\Dom(f)\leq \mathfrak{g}$. Additionally, by the linearity of $f$ and of the Lie bracket,
        \begin{align*}
            f([g+g',g_1])&=f([g,g_1])+f([g',g_1])=gf(g_1)-g_1f(g)+g'f(g_1)-g_1f(g')\\
            &=(g+g')f(g_1)-g_1f(g+g').
        \end{align*}
        Hence $D^f_{g+g'}\geq D^f_g\cap D^f_{g'}$. Since the latter has finite index in $\mathfrak{g}$, we conclude that $g+g'\in D^f$. The closure under the additive inverse follows similarly.\\
        (3) Let $f,h\in \operatorname{AlmDer}(\mathfrak{g},A)$. Then $D^f\cap D^h$ has finite index in $\mathfrak{g}$. If we verify that $g\in D^{f+h}$ for every $g\in D^f\cap D^h$, then the claim follows. Fix $g\in D^f\cap D^h$, then $D^f_g\cap D^h_g$ has finite index in $\mathfrak{g}$. In addition, for every $g_1\in D^f_g\cap D^h_g$, the element $(f+h)([g,g_1])$ is defined since $\Dom(f+h)=\Dom(f)\cap \Dom(h)$. The same holds for $(f+h)(g)$ and $(f+h)(g_1)$. Moreover, 
        \begin{align*}
            (f+h)([g,g_1])&=f([g,g_1])+h([g,g_1])\\
            &=gf(g_1)-g_1f(g)+gh(g_1)-g_1h(g)=g(f+h)(g_1)-g_1(f+h)(g).
        \end{align*}
        The proof for the additive inverse follows similarly.\\
        (4) Let $f$ be an almost derivation, and let $h\in \mathcal{P}\text{-}\operatorname{Hom}(\mathfrak{g},A)$ be such that $(f-h)(\Dom(f)\cap \Dom(h))$ is finite. Since they are both partial homomorphisms, the subgroup $\ker(f-h)$ has finite index in $\Dom(f)\cap \Dom(h)$ and so in $\mathfrak{g}$. If we verify that, for every $g\in D^f\cap \ker(f-h)$, the element $g$ belongs to $D^h$, the claim follows. For every $g_1\in D^f_g\cap \ker(f-h)\cap \ad_g^{-1}(\ker(f-h))$, we have that $h(g_1)$ and $h([g,g_1])$ are well-defined. Moreover, 
        $$h([g,g_1])=f([g,g_1])=gf(g_1)-g_1f(g)=gh(g_1)-g_1h(g).$$
        Since $D^f_g\cap \ker(f-h)\cap \ad_g^{-1}(\ker(f-h))$ has finite index in $\mathfrak{g}$, we conclude that $g\in D^h$ and the proof is completed.
    \end{proof}
    We also define the subgroup of \emph{inner derivation} of the module $(\mathfrak{g},A)$.
    \begin{definition}
    Let $(\mathfrak{g},A)$ be a module. Then $f\in \mathcal{P}\text{-}\operatorname{Hom}(\mathfrak{g},A)$ is an \emph{inner derivation} if $f=d(a)$ for some $a\in A$.
    \end{definition}
        
    We now characterise the $0$th and $1$st almost cohomology groups.
    \begin{lemma}
    Let $(\mathfrak{g},A)$ be a $\mathfrak{g}$-module. Then
    \begin{enumerate}
        \item $\widetilde{H}^0(\mathfrak{g},A)=\widetilde{C}_A(\mathfrak{g})$;
        \item $\widetilde{H}^1(\mathfrak{g},A)=\operatorname{AlmDer}(\mathfrak{g},A)/[\operatorname{InnDer}(\mathfrak{g},A)]$.
    \end{enumerate}
    \end{lemma}
    \begin{proof}
        (1) Note that $\widetilde{H}^0(\mathfrak{g},A)=\ker(d)$, and $a\in \ker(d_0)$ if and only if the map $d(a):\mathfrak{g}\to A$ that sends $g\in \mathfrak{g}$ to $g(a)\in A$ is equivalent to the zero map. Therefore $a\in \ker(d)$ if and only if $\{g\in \mathfrak{g}:\ g(a)=0\}=C_{\mathfrak{g}}(a)\sqsubseteq \mathfrak{g}$ has finite index in $\mathfrak{g}$, and so $a\in \widetilde{C}_A(\mathfrak{g})$.\\
        (2) The image of $d$ consists of the $f\in \widetilde{C}^1(\mathfrak{g},A)$ such that $f-d(a)$ is finite for some $a\in A$. This is equivalent to saying that $f$ is equivalent to the inner derivation $d(a)$.
    \end{proof}
    Both $\widetilde{H}^0(\mathfrak{g},A)$ and $\widetilde{H}^1(\mathfrak{g},A)$ can be equipped with the structure of a $\mathfrak{g}$-module.
    \begin{definition}
         We define the maps
    $$\cdot:\mathfrak{g}\times \widetilde{C}^0(\mathfrak{g},A)\to \widetilde{C}^0(\mathfrak{g},A)$$
    that sends the pair $(g,a)$ to $ga$, and
    $$\cdot:\mathfrak{g}\times \widetilde{C}^1(\mathfrak{g},A)\to \widetilde{C}^1(\mathfrak{g},A)$$
    that sends the pair $(g,[f])$ to $[g\cdot f]$, where 
    $$(g\cdot f):\Dom(f)\cap \ad_g^{-1}(\Dom(f))\to A$$
    is the map sending $g'\in \Dom(f)\cap \ad_g^{-1}(\Dom(f))$ to $(g\cdot f)(g')=g(f(g'))-f([g,g'])$.
    \end{definition}
    Note that the map $\cdot:\mathfrak{g}\times \widetilde{C}^1(\mathfrak{g},A)\to \widetilde{C}^1(\mathfrak{g},A)$ is well-defined because if $f\sim 0$, then $\mathrm{im}(g\cdot f)\leq \mathrm{im}(g\cdot f)+\mathrm{im}(f)$, which is finite since it is a finite sum of finite subgroups.
    We now show that the actions of $\mathfrak{g}$ on $\widetilde{C}^0(\mathfrak{g},A)$ (respectively $\widetilde{C}^1(\mathfrak{g},A)$) induces an action on $\widetilde{H}^0(\mathfrak{g},A)$ (respectively $\widetilde{H}^1(\mathfrak{g},A)$). Moreover, we establish that the action of $\mathfrak{g}$ on $\widetilde{H}^1(\mathfrak{g},A)$ is almost trivial.
    \begin{lemma}\label{actalmcoh}
    Let $\mathfrak{g}$ be a Lie ring and $(\mathfrak{g},A)$ a $\mathfrak{g}$-module.
    \begin{enumerate}
        \item The action of $\mathfrak{g}$ on $\widetilde{C}^0(\mathfrak{g},A)$ induces an action of $\mathfrak{g}$ on $\widetilde{H}^0(\mathfrak{g},A)$;
        \item The action of $\mathfrak{g}$ on $\widetilde{C}^1(\mathfrak{g},A)$ induces an action of $\mathfrak{g}$ on $\widetilde{H}^1(\mathfrak{g},A)$. Furthermore $\widetilde{C}_{\widetilde{H}^1(\mathfrak{g})}(\mathfrak{g})=\widetilde{H}^1(\mathfrak{g})$.
    \end{enumerate}
    \end{lemma}
    \begin{proof}
        (1) We verify that the restriction of $\cdot$ to $\mathfrak{g}\times \widetilde{H}^0(\mathfrak{g},A)$ has image in $\widetilde{H}^0(\mathfrak{g},A)$. Let $a\in \widetilde{H}^0(\mathfrak{g},A)$, then $a\in \widetilde{C}_A({\mathfrak{g}})$. Since $\widetilde{C}_A({\mathfrak{g}})$ is $\mathfrak{g}$-invariant by Lemma \ref{DefZ}, we obtain that $ga\in g\widetilde{C}_A(\mathfrak{g})\leq \widetilde{C}_A(\mathfrak{g})$.\\
        (2) We show that, for every $f\in \operatorname{AlmDer}(\mathfrak{g},A)$ and $g\in \mathfrak{g}$, the partial homomorphism $g\cdot f$ belongs to $\operatorname{AlmDer}(\mathfrak{g},A)$. Let $g'\in \Dom(f)\cap \ad_{g}^{-1}(\Dom(f))$. Then the subgroup $D^{g\cdot f}_{g'}$ coincides with
        $$D^{g\cdot f}_{g'}=\{g''\in \Dom(g\cdot f)\cap \ad^{-1}_{g'}(\Dom(g\cdot f)):\ (g\cdot f)([g',g''])=g'(g\cdot f)(g'')-g''(g\cdot f)(g')\}.$$
        By definition of the action, $(g\cdot f)([g',g''])=gf([g',g''])-f([g,[g',g'']])$, while $(g\cdot f)(g'')=g(f(g''))-f([g,g''])$ and $(g\cdot f)(g')=g(f(g'))-f([g,g'])$. Hence, 
     $$g'(g\cdot f)(g'')-g''(g\cdot f)(g')=g'g(f(g''))-g'f([g,g''])-g''gf(g')+g''f([g,g']).$$
     We now evaluate $(g\cdot f)([g',g''])=g(f([g',g'']))-f([g,[g',g'']])$. By the Jacobi identity and the linearity of $f$, 
     $$-f([g,[g',g'']])=f([g'',[g,g']])-f([g',[g,g'']).$$
       By assumption, the subgroup $D^f\leq \mathfrak{g}$ has finite index in $\mathfrak{g}$, and so also $D^f\cap \ad_{g}^{-1}(D^f)$ has finite index in $\mathfrak{g}$. Take $g'\in D^f\cap \ad_{g}^{-1}(D^f)$ and $g''\in \ad_g^{-1}(D_{g'}^f)\cap D^f_{[g,g']}\cap D^f_{g'}$. The subgroup $\ad_g^{-1}(D_{g'}^f)\cap D^f_{[g,g']}\cap D_{g'}^f$ has finite index by the assumption on $g'$. Let $g''\in \ad_g^{-1}(D_{g'}^f)\cap D^f_{[g,g']}\cap D_{g'}^f$. Then 
       $$f([g'',[g,g']])-f([g',[g,g''])=g''f([g,g'])-[g,g']f(g'')-g'f([g,g''])+[g,g'']f(g').$$
       Thus, we have that $(g\cdot f)([g',g''])=g'(g\cdot f)(g'')-g''(g\cdot f)(g')$ if and only if
       \begin{align*}
          &gf([g',g''])+f([g'',[g,g']])-f([g',[g,g''])\\
          &=gf([g',g''])+g''f([g,g'])-[g,g']f(g'')-g'f([g,g''])+[g,g'']f(g')\\
          &=g'g(f(g''))-g'f([g,g''])-g''gf(g')+g''f([g,g']).
       \end{align*}
       Canceling the terms that appear on both sides, and since $[g,g']f(g'')=gg'f(g'')-g'gf(g'')$ and $[g,g'']f(g')=gg''f(g')-g''gf(g')$, we obtain that
       $$gf([g',g''])=gg'f(g'')-gg'f(g').$$
    This identity holds because $g''\in D^f_{g'}$. Consequently, we have verified that $\mathfrak{g}$ acts on $\operatorname{AlmDer}(\mathfrak{g},A)$.\\
    We now show that, if $f\sim da$ for some $a\in A$, then $gf\sim d(ga)$. Let $f\in \widetilde{C}^1(\mathfrak{g},A)$ and $a\in A$ be such that $f\sim d(a)$. Let $H\leq \mathfrak{g}$ be the kernel of $f-d(a)$. Then, for every $g'\in H\cap \ad_g^{-1}(H)$, 
    $$(g\cdot f)(g')=gf(g')-f([g,g'])=g(g'(a))-[g,g'](a)=g'(ga).$$
    Since $H\cap \ad^{-1}_g(H)$ has finite index in $\mathfrak{g}$, we obtain that $gf\sim d(ga)$. This verifies that $\mathfrak{g}$ acts on $\widetilde{H}^1(\mathfrak{g},A)$.\\
    We finally prove that the action of $\mathfrak{g}$ on $\widetilde{H}^1(\mathfrak{g},A)$ is almost trivial. Let $f\in \operatorname{AlmDer}(\mathfrak{g},A)$. We verify that, for every $g\in D^f$, the map $g\cdot f$ is equivalent to $d(f(g))$. For every $g'\in D^f_g$,
    $$(g\cdot f)(g')=gf(g')-f([g,g'])=g'(f(g)).$$
    Since $D^f_g$ has finite index in $\mathfrak{g}$, we conclude that $gf\sim d(f(g))$.
    \end{proof}
    \section{Snake Lemma}
    We now study the relation between the almost cohomology groups of $\mathfrak{g}$-modules in an exact sequence. Note that for an exact sequence 
    $$0\xrightarrow{} A'\xrightarrow{i} A\xrightarrow{p} A''\xrightarrow{} 0$$
    of $\mathfrak{g}$-modules, the techniques developed in \cite{zamourCohom} cannot be applied. Indeed,
    \begin{itemize}
        \item There are neither second almost cohomology groups nor a well-defined $\widetilde{C}^2(\mathfrak{g},A)$. Therefore, we cannot apply the Snake Lemma (\cite[Theorem 1.3.1]{weibel1994introduction}).
        \item In general, a section of $p$ (i.e. a map $s:A''\to A$ such that $p\circ s=\operatorname{Id}_{A''}$) does not exist. Consequently, the sequence 
        $$0\xrightarrow{} \widetilde{C}^1(\mathfrak{g},A')\xrightarrow{i^1} \widetilde{C}^1(\mathfrak{g},A)\xrightarrow{p^1} \widetilde{C}^1(\mathfrak{g},A'')\xrightarrow{} 0$$
         need to split.
    \end{itemize}
    Nevertheless, we show that one still obtains the following "long" exact sequence in the almost cohomology:
    $$0\to \widetilde{H}^0(\mathfrak{g},A')\xrightarrow{i} \widetilde{H}^0(\mathfrak{g},A)\xrightarrow{p} \widetilde{H}^0(\mathfrak{g},A'')\xrightarrow{\delta_1} \widetilde{H}^1(\mathfrak{g},A')\xrightarrow{i^1}\widetilde{H}^1(\mathfrak{g},A)\xrightarrow{p^1} \widetilde{H}^1(\mathfrak{g},A'').$$
    The following lemma provides, for an exact sequence of $\mathfrak{g}$-modules, two exact sequences that play a fundamental role in the construction of the sequence in almost cohomology.
    \begin{lemma}\label{ExactserC}
    Let $0\xrightarrow{} A'\xrightarrow{i} A \xrightarrow{p} A''\xrightarrow{} 0$ be a short exact sequence of $\mathfrak{g}$-modules. Then it induces the exact sequences $$0\xrightarrow{} \widetilde{C}^0(\mathfrak{g},A')\xrightarrow{i^0} \widetilde{C}^0(\mathfrak{g},A) \xrightarrow{p^0} \widetilde{C}^0(\mathfrak{g},A'')\xrightarrow{} 0$$
    and
    $$0\xrightarrow{} \widetilde{C}^1(\mathfrak{g},A')\xrightarrow{i^1} \widetilde{C}^1(\mathfrak{g},A) \xrightarrow{p^1} \widetilde{C}^1(\mathfrak{g},A'').$$
    \end{lemma}
    \begin{proof}
    We define $i^0=i$ and $p^0=p$. The first sequence is exact since it coincides with the sequence of $\mathfrak{g}$-modules, which is exact by hypothesis. We define $i^1$ by $i^1([f])=[i\circ f]$, and $p^1$ similarly.
    \begin{itemize}
        \item We first verify that the map $i^1$ is well-defined. If $[f]=[f_1]\in \widetilde{C}^1(\mathfrak{g},A)$, then $B:=(f-f_1)(\Dom(f)\cap \Dom(f_1))$ is finite. Since $i$ is a group homomorphism, $i(f(g))-i(f_1(g))=i(f(g)-f_1(g))\in i(B)$ for every $g\in \Dom(f)\cap \Dom(f_1)$. As $i(B)$ is finite, $i\circ f\sim i\circ f_1$ and so the map $i^1$ is well-defined. The same argument applies to $p^1$.
        \item  Since $i$ and $p$ commute with the action of $\mathfrak{g}$, we have that 
        $$d(i(a))(g)=g(i(a))=i(g(a))=(i^1(d(a)))(g).$$
        This shows that $d\circ i=i^1\circ d$. The same holds for $p$.
        \item We prove that $i^1$ is injective. Let $[f]\in \widetilde{C}^1(\mathfrak{g},A')$ be such that $i^1[f]=[0]$. Thus, the subgroup $i\circ f(\Dom(f))$ is finite. Since $i$ is injective, also $f(\Dom(f))$ is finite, and so $[f]=[0]$.
        \item We establish that $\mathrm{im}(i^1)\subseteq \operatorname{ker}(p^1)$. Let $[f]=i^1([h])\in \widetilde{C}^1(\mathfrak{g},A'')$ for $h\in \widetilde{C}^1(\mathfrak{g},A')$. Then $(f-i\circ h)(\Dom(f)\cap \Dom(i\circ h))$ is finite. Consequently, there exists a subgroup $G_1$ of finite index in $\mathfrak{g}$ such that $f(g)=i\circ h(g)$ for every $g\in G_1$. Then, for every $g\in G_1$, 
    $$(p\circ f)(g)=(p\circ (i\circ h))(g)=0$$
    by assumption. Therefore $\mathrm{im}(i^1)\subseteq\ker(p^1)$.
    \item We show that $\mathrm{im}(i^1)\supseteq \operatorname{ker}(p^1)$. Let $f\in \widetilde{C}^1(\mathfrak{g},A)$ be such that $p^1[f]=[0]$. Then the map $p\circ f:\Dom(f)\to A''$ has finite image. Consequently, the subgroup $D=\{g\in \Dom(f):\ p(f(g))=0\}$
    has finite index in $\mathfrak{g}$. Since $\ker(p)=\mathrm{im}(i)$ by hypothesis, the map $h\in \widetilde{C}^1(\mathfrak{g},A')$ defined by $h(g)=a$ with $i(a)=f(g)$ for $g\in D$ is well-defined (since $i$ is injective) and belongs to $\widetilde{C}^1(\mathfrak{g},A')$. Since $\mathrm{im}(f-i\circ h)$ is finite, the claim follows.
    \end{itemize}
    \end{proof}
Note that Lemma \ref{ExactserC} does not hold if we do not assume that the domain of the elements in $\widetilde{C}^1(\mathfrak{g},A)$ can be a proper subgroup of finite index in $\mathfrak{g}$. Indeed, we cannot prove that $\mathrm{im}(i^1)\supseteq \operatorname{ker}(p^1)$.\\
The map $p^1$ is surjective whenever $p$ has a section $s:A''\to A$. Indeed, given $f\in \widetilde{C}^1(\mathfrak{g},A'')$, let $s^1(f)\in \widetilde{C}^1(\mathfrak{g},A)$ be the map defined by $f_1(g)=s(f(g))$ for every $g\in \Dom(f)$. It is immediate to verify that, for every $f\in \widetilde{C}^1(\mathfrak{g},A'')$, we have that $p^1(s^1([f]))=[f]$.\\
We now apply the strategy of the Snake Lemma to construct a "long" exact sequence in almost cohomology.
    \begin{lemma}\label{longseries}
        Let 
        $$0\xrightarrow{} A'\xrightarrow{i} A \xrightarrow{p} A''\xrightarrow{} 0$$
        be an exact sequence of $\mathfrak{g}$-modules. Then we can construct a homomorphism of $\mathfrak{g}$-modules
        $$\delta_1:\widetilde{H}^0(\mathfrak{g},A'')\to \widetilde{H}^1(\mathfrak{g},A')$$
        such that the sequence
        $$0\to \widetilde{H}^0(\mathfrak{g},A')\xrightarrow{i} \widetilde{H}^0(\mathfrak{g},A)\xrightarrow{p} \widetilde{H}^0(\mathfrak{g},A'')\xrightarrow{\delta_1} \widetilde{H}^1(\mathfrak{g},A')\xrightarrow{i^1} \widetilde{H}^1(\mathfrak{g},A)\xrightarrow{p^1}\widetilde{H}^1(\mathfrak{g},A'')$$
         is exact.
    \end{lemma}
   \begin{proof}
   We construct a homomorphism of $\mathfrak{g}$-modules
   $$\delta_1:\widetilde{H}^0(A'')\to \widetilde{H}^1(A')$$
   such that $\ker(\delta_1)=\mathrm{im}(p_0)$ and $\mathrm{im}(\delta_1)=\operatorname{ker}(i_1)$.\\
   Denote by $d''$ (respectively $d',d$) the differential for the module $(\mathfrak{g},A'')$ (respectively $(\mathfrak{g},A'),(\mathfrak{g},A)$). 
   The following diagram represents the situation we are working in.\\
   \begin{center}{
   \begin{tikzcd}
0\arrow[r]& \widetilde{C}^0(\mathfrak{g},A') \arrow[r,"i"]\arrow[d,"d'"]
& \widetilde{C}^0(\mathfrak{g},A) \arrow[r,"p"]
\arrow[d, phantom, ""{coordinate, name=Z}]\arrow[d,"d"]
& \widetilde{C}^0(\mathfrak{g},A'') \arrow[dll,
"\delta_1",
rounded corners,
to path={ -- ([xshift=2ex]\tikztostart.east)
|- (Z) [near end]\tikztonodes
-| ([xshift=-2ex]\tikztotarget.west)
-- (\tikztotarget)}] \arrow[r]\arrow[d,"d''"]&0\\
0\arrow[r]&\widetilde{C}^1(\mathfrak{g},A') \arrow[r,"i^1"]
& \widetilde{C}^1(\mathfrak{g},A) \arrow[r,"p^1"]
& \widetilde{C}^1(\mathfrak{g},A'')
\end{tikzcd}}
\end{center}
   Fix $a''\in \ker(d'')$. Since $p$ is surjective, there exists some $a\in A$ such that $p(a)=a''$. As $p^1\circ d=d\circ p$, we have that $p^1([d(a)])=[d(pa)]=[d(a'')]=[0]$. Since $\mathrm{im}(i^1)=\ker(p^1)$ by Lemma \ref{ExactserC}, there exists a unique $[f_{a''}]\in \widetilde{C}^1(\mathfrak{g},A')$ such that $i^1([f_{a''}])=[d(a)]$. We define $\delta_1:\ker(d'')=\widetilde{H}^0(\mathfrak{g},A'')\to \widetilde{C}^1(\mathfrak{g},A')/\mathrm{im}(d')$ as the function that sends $a''$ to $[f_{a''}]+\mathrm{im}(i^1)$.
   \begin{itemize}
       \item We first check that this map is well-defined. Let $a_1\in A$ be such that $p(a_1)=a''$. Then $p(a-a_1)=0$. Since $\ker(p)=\mathrm{im}(i)$ and by the injectivity of $i$, there exists a unique $b'\in A'$ such that $a-a_1=i(b')$. As $d\circ i=i^1\circ d'$, we have that $d(a)-d(a_1)=d(a-a_1)=d(i(b'))=i^1(d'(b'))$. Hence $i^{-1}(d(a)-d(a_1))=d'(b')\in \mathrm{im}(d')$. Therefore, the map $\delta_1$ is well-defined.
       \item  A standard calculation shows that $\delta_1$ is a homomorphism of $\mathfrak{g}$-modules.
       \item We prove that $f_{a''}\in \operatorname{AlmDer}(\mathfrak{g},A')$ for every $a''\in \widetilde{H}^0(\mathfrak{g},A'')$. Let $g\in \Dom(f_{a''})$ and $g'\in \Dom(f_{a''})\cap \ad_g^{-1}(\Dom(f_{a''})$. Then $f_{a''}([g,g'])=gf_{a''}(g')-g'f_{a''}(g)$ if and only if, by the injectivity of $i$, we have that $i(f_{a''}([g,g']))=i(gf_{a''}(g'))-i(g'f_{a''}(g))$. By the definition of $f_{a''}$, the previous identity holds if and only if $[g,g'](a)=g(g'a)-g'(ga)$, which is clearly true.
   \end{itemize}
   This shows that $\delta_1:\widetilde{H}^0(\mathfrak{g},A'')\to \widetilde{H}^1(\mathfrak{g},A')$ is a homomorphism of $\mathfrak{g}$-modules.\\
   We now prove that the sequence in almost cohomology is exact.
   \begin{itemize}
       \item The map $i^0:\widetilde{H}^0(\mathfrak{g},A')\to \widetilde{H}^0(\mathfrak{g},A)$ is well-defined. Indeed, for every $a'\in \widetilde{C}_{A'}(\mathfrak{g})$ and $g\in C_{\mathfrak{g}}(a')$, it holds that $g(i(a'))=i(g(a'))=i(0)=0$. Since $C_{\mathfrak{g}}(a')\sim \mathfrak{g}$, the same holds for $i(a')$, and so $i(a)\in \widetilde{H}^0(\mathfrak{g},A)$.
       \item For the same reasons, the map $p:\widetilde{H}^0(\mathfrak{g},A)\to \widetilde{H}^0(\mathfrak{g},A'')$ is well-defined.
       \item $i^0$ is clearly injective since the same holds for $i:A'\to A$. 
       \item We verify that $\mathrm{im}(i)\subseteq \ker(p)$. Let $a\in A$ be such that $i(a')=a$ for some $a'\in A'$. Then, by the exactness of the sequence of $\mathfrak{g}$-modules, we have that $p(a)=0$.
       \item We prove that $\mathrm{im}(i)\supseteq \ker(p)$. Fix $a\in \widetilde{C}_A(\mathfrak{g})$ such that $p(a)=0$. Then, by the exactness of the series of $\mathfrak{g}$-modules, there exists some $a'\in A'$ such that $i(a')=a$. For every $g\in C_{\mathfrak{g}}(a)$, we have that $0=g(a)=g(i(a'))=i(g(a'))$, and so $g\in C_{\mathfrak{g}}(a')$ by the injectivity of $i$. Hence $a'\in \widetilde{H}^0(\mathfrak{g},A')$. This proves that $\ker(p^0)\subseteq \mathrm{im}(i^0)$.
       \item The map $p^1:\widetilde{H}^1(\mathfrak{g},A)\to\widetilde{H}^1(\mathfrak{g},A'')$ is well-defined. Let $f\in \operatorname{AlmDer}(\mathfrak{g},A)$ and $g\in \Dom(f)$. Then, for every $g'\in D^f_g$, 
   $$p^1f([g,g'])=p(gf(g')-g'f(g))=gp(f(g'))-g'pf(g)=g(p^1f)(g')-g'(p^1f)(g).$$
   Therefore, $p^1(f)$ is an almost derivation.
   \item $i^1:\widetilde{H}^1(\mathfrak{g},A')\to\widetilde{H}^1(\mathfrak{g},A)$ is well-defined by a similar argument.
   \item We show that $\ker(p^1)\subseteq\mathrm{im}(i^1)$. Let $f\in \ker(p^1)$. Then $p^1(f)\sim d''(a'')$ for some $a''\in A''$. By the surjectivity of $p$ and since $d''\circ p=p^1\circ d$, there exists some $a\in A$ such that $d''(a'')=d''(p(a))=p^1(d(a))$. Therefore $p^1(f-da)$ is finite. By Lemma \ref{ExactserC}, there exists some $h\in \widetilde{C}^1(\mathfrak{g},A'')$ such that $(f-da)\sim i^1(h)$. Since $i^1$ is injective and $f-da$ is an almost derivation, also $h$ is an almost derivation. Consequently, we have that $[f]=i^1([h])$ in $\widetilde{H}^1(\mathfrak{g},A)$.
   \item We prove that $\ker(p^1)\supseteq\mathrm{im}(i^1)$. Let $i^1([h])\in \widetilde{H}^1(\mathfrak{g},A)$. Then 
   $$p^1(i^1[h])=p^1[i\circ h]=[p\circ (i\circ h)]=[0].$$
   \item We verify that $\ker(\delta_1)=\mathrm{im}(p)$. Let $a''\in \mathrm{im}(p)$. Then $a''=p(a)$ for some $a\in A$ with $[d(a)]=[0]$, and so $\delta_1(a'')=i^{-1}[d(a)]=[0]$. This proves that $\ker(\delta_1)\supseteq \mathrm{im}(p)$. For the converse, let $a''\in \ker(\delta_1)$, and let $a\in A$ be such that $p(a)=a''$. Then $i^1[d(a)]=[0]$ and so $[d(a)]=0$.
   \item Finally, we show that $\mathrm{im}(\delta_1)=\ker(i^1)$. Let $f_{a''}=\delta_1(a'')$. Then $i^1(f_{a''})=[d(a)]$ for some $a\in A$ such that $p(a)=a''$. Hence $[i^1(f)]=[0]$ in $\widetilde{H}^1(\mathfrak{g},A)$. This proves that $\mathrm{im}(\delta_1)\subseteq\ker(i^1)$. For the converse, let $[f]\in\ker(i^1)$. Then $i^1[f]=[da]$. Let $a''\in A''$ be such that $a''=p(a)$. Then $a''\in\ker(d)$ since $d''(a'')=d''(p(a))=p^1(d(a))=p^1(i^1(f))=0$. we conclude that $\delta_1(a'')=f$. 
   \end{itemize}
   \end{proof}
   Two immediate corollaries of Lemma \ref{longseries} allow us to characterise the almost cohomology groups of $(\mathfrak{g},A)$ from the almost cohomology groups of an isogenous $\mathfrak{g}$-module. We begin by showing that, for a module $(\mathfrak{g},A)$ and a $\mathfrak{g}$-submodule of finite index $A_1$ in $A$, the groups $\widetilde{H}^0(\mathfrak{g},A)$ and $\widetilde{H}^1(\mathfrak{g},A)$ are isogenous to $\widetilde{H}^0(\mathfrak{g},A_1)$ and $\widetilde{H}^1(\mathfrak{g},A_1)$, respectively.
   \begin{corollary}\label{A/A_1}
   Let $(\mathfrak{g},A)$ be a $\mathfrak{g}$-module, and let $A_1\leq A$ be a definable $\mathfrak{g}$-submodule of finite index. Then
   \begin{enumerate}
       \item $\widetilde{H}^0(\mathfrak{g},A)$ is isogenous to $ \widetilde{H}^0(\mathfrak{g},A_1)$;
       \item $\widetilde{H}^1(\mathfrak{g},A)$ is isogenous to $ \widetilde{H}^1(\mathfrak{g},A_1)$.
   \end{enumerate}
   \end{corollary}
   \begin{proof}
   (1) The group $\widetilde{H}^0(\mathfrak{g},A_1)=\widetilde{C}_{A_1}(\mathfrak{g})=A_1\cap \widetilde{C}_{A}(\mathfrak{g})$ is isogenous to $\widetilde{C}_A(\mathfrak{g})$ since $A_1$ has finite index in $A$.\\
    (2) The sequence 
    $$0\to A_1\xrightarrow{i}A\xrightarrow{p}A/A_1\xrightarrow{}0,$$
    where $i$ is the inclusion of $A_1$ in $A$ and $p$ is the projection of $A$ to $A/A_1$, is exact. By Lemma \ref{longseries}, the sequence 
    $$\widetilde{H}^0(\mathfrak{g},A/A_1)\xrightarrow[]{\delta_1} \widetilde{H}^{1}(\mathfrak{g},A_1)\xrightarrow{i^1} \widetilde{H}^1(\mathfrak{g},A)\xrightarrow{p^1}\widetilde{H}^1(\mathfrak{g},A/A_1)$$
    is exact. Since $A/A_1$ is finite, also the group $\widetilde{H}^0(\mathfrak{g},A/A_1)$ is finite. Moreover $\widetilde{H}^1(\mathfrak{g},A/A_1)$ is trivial, since every partial homomorphism $f\in \mathcal{P}\text{-}\operatorname{End}(\mathfrak{g},A/A_1)$ has kernel of finite index in $\mathfrak{g}$. Consequently, the map
    $$i^1:\widetilde{H}^1(\mathfrak{g},A_1)\to \widetilde{H}^1(\mathfrak{g},A)$$
    is a surjective map with finite kernel, since this kernel coincides with the image of $\delta_1$. Thus, the group $\widetilde{H}^1(\mathfrak{g},A_1)$ is isogenous to $\widetilde{H}^1(\mathfrak{g},A)$.
   \end{proof}
   Let $(\mathfrak{g},A)$ be a $\mathfrak{g}$-module, and let $A_2\leq A$ be a finite $\mathfrak{g}$-submodule. We prove that $\widetilde{H}^0(\mathfrak{g},A)$ is isogenous to $\widetilde{H}^0(\mathfrak{g},A/A_2)$ and $\widetilde{H}^1(\mathfrak{g},A)$ is isomorphic to a subgroup of $\widetilde{H}^1(\mathfrak{g},A/A_2)$.
   \begin{corollary}\label{A/A_2}
   Let $(\mathfrak{g},A)$ be a $\mathfrak{g}$-module, and let $A_2$ be a finite $\mathfrak{g}$-submodule of $A$. Then
   \begin{enumerate}
       \item $\widetilde{H}^0(\mathfrak{g},A)$ is isogenous to $\widetilde{H}^0(\mathfrak{g},A/A_2)$;
       \item $\widetilde{H}^1(\mathfrak{g},A)$ is isomorphic to a subgroup of $\widetilde{H}^1(\mathfrak{g},A/A_2)$.
   \end{enumerate}
   \end{corollary}
   \begin{proof}
       (1) The group $\widetilde{H}^0(\mathfrak{g},A/A_2)$ coincides with $\widetilde{C}_{A/A_2}(\mathfrak{g}):=\{a+A_2\in A/A_2:\ \mathfrak{g}(a)\apprle A_2\}$. Since $A_2$ is finite and contained in $\widetilde{C}_A(\mathfrak{g})$, we have that $\widetilde{H}^0(\mathfrak{g},A/A_2)$ equals $\widetilde{C}_{A}(\mathfrak{g})/A_2$. As $A_2$ is finite, point (1) follows.\\
       (2) The sequence
       $$0\to A_2\xrightarrow{i}A\xrightarrow{p}A/A_2\xrightarrow{}0,$$
       where $i$ the inclusion of $A_2$ in $A$ and $p$ is the projection of $A$ to $A/A_2$, is exact. By Lemma \ref{longseries}, the sequence
       $$\widetilde{H}^{1}(\mathfrak{g},A_2)\xrightarrow{i^1} \widetilde{H}^1(\mathfrak{g},A)\xrightarrow{p^1}\widetilde{H}^1(\mathfrak{g},A/A_2)$$
       is exact. Since $A_2$ is finite, we obtain that $\widetilde{H}^1(\mathfrak{g},A_2)$ is trivial by the same argument of Lemma \ref{A/A_1}(2). Consequently, the map $p^1:\widetilde{H}^1(\mathfrak{g},A)\to \widetilde{H}^1(\mathfrak{g},A/A_2)$ is injective and the claim follows.
   \end{proof}
   \section{Almost cohomology groups of a quotient}
   We now analyse the almost cohomology groups of the module $(\mathfrak{g}/\mathfrak{h},A)$ where $\mathfrak{h}$ is a definable ideal of a finite-dimensional Lie ring $\mathfrak{g}$.
   We begin by defining the maps $inf$ and $res$
   \begin{definition}
   Let $\mathfrak{g}$ be a Lie ring acting on the abelian group $A$, both definable in a finite-dimensional theory. Let $\mathfrak{h}\mathrel{\unlhd}\mathfrak{g}$ be an ideal of $\mathfrak{g}$. Let $\pi:\mathfrak{g}\to \mathfrak{g}/\mathfrak{h}$ be the natural projection. Then we define the map 
   $$inf:\widetilde{H}^1(\mathfrak{g}/\mathfrak{h},A)\to \widetilde{H}^1(\mathfrak{g},A)$$
   that sends $[f]\in \widetilde{H}^1(\mathfrak{g}/\mathfrak{h},A)$ to $inf[f]=[inf(f)]\in \widetilde{H}^1(\mathfrak{g},A)$, where $inf(f):\pi^{-1}(\Dom(f))\to A$
   is the partial homomorphism that sends $g\in \pi^{-1}(\Dom(f))$ to $f(g+\mathfrak{h})\in A$.\\
   We define
   $$res:\widetilde{H}^1(\mathfrak{g},A)\to \widetilde{H}^1(\mathfrak{h},A)$$
   as the map that sends $[f]\in \widetilde{H}^1(\mathfrak{g},A)$ to $res[f]=[res(f)]$, where $res(f):\mathfrak{h}\cap \Dom(f)\to A$ is the partial homomorphism
   that sends $h\in \Dom(f)\cap \mathfrak{h}$ to $f(h)\in A$.
   \end{definition}
   The proof of the well-definedness of these maps follow from standard calculations.
   \begin{lemma}\label{exact series}
   Let $\mathfrak{g}$ be a Lie ring acting on the abelian group $A$, both definable in a finite-dimensional theory. Assume that $\widetilde{C}_A(\mathfrak{g})$ is trivial. Let $\mathfrak{h}\mathrel{\unlhd}\mathfrak{g}$ be a definable ideal and denote by $\mathfrak{h}_1$ the $\mathfrak{g}$-ideal $\widetilde{C}_{\mathfrak{h}}(\widetilde{C}_A(\mathfrak{h}))$.
   We define the map 
   $$inf':\widetilde{H}^1(\mathfrak{g}/\mathfrak{h}_1,\widetilde{C}_A(\mathfrak{h}))\to \widetilde{H}^1(\mathfrak{g},A)$$
   given by the composition of $inf$ and of the map $i^1:\widetilde{H}^1(\mathfrak{g},\widetilde{C}_{\mathfrak{h}}(A))\to \widetilde{H}^1(\mathfrak{g},A)$ defined in Lemma \ref{longseries}.
   Then the sequence 
   $$0\xrightarrow{} \widetilde{H}^1(\mathfrak{g}/\mathfrak{h}_1,\widetilde{C}_A(\mathfrak{h}))\xrightarrow{inf'} \widetilde{H}^1(\mathfrak{g},A)\xrightarrow{res} \widetilde{H}^1(\mathfrak{h},A)$$
   is exact.
   \end{lemma}
   \begin{proof}
   We begin by verifying that $\mathfrak{g}/\mathfrak{h}_1$ acts on $\widetilde{C}_A(\mathfrak{h})$. Since $\mathfrak{h}$ is an ideal in $\mathfrak{g}$, the subgroup $\widetilde{C}_A(\mathfrak{h})$ is a definable $\mathfrak{g}$-module by Lemma \ref{IdeCenAct}. Hence, the Lie subring $\mathfrak{h}_1=\widetilde{C}_{\mathfrak{g}}(\widetilde{C}_A(\mathfrak{h}))$ is a definable ideal in $\mathfrak{g}$ by Lemma \ref{DefZ}. Moreover $\mathfrak{h}_1$ has finite index in $\mathfrak{h}$ by Lemma \ref{symact}. To prove the claim, it suffices to show that $h(\widetilde{C}_A(\mathfrak{h}))\leq \widetilde{C}_A({\mathfrak{g}})=\{0\}$ for every $h\in \mathfrak{h}_1$. By Lemma \ref{fin[]}, the subgroup $[\mathfrak{h}_1,\widetilde{C}_A({\mathfrak{h}})]$ is finite. Furthermore $[\mathfrak{h}_1,\widetilde{C}_A({\mathfrak{h}})]$ is $\mathfrak{g}$-invariant since, for every $h\in \mathfrak{h}_1$ and $a\in \widetilde{C}_A({\mathfrak{h}})$, 
   $$g(h(a))=[g,h](a)+h(g(a))\in [\mathfrak{h}_1,\widetilde{C}_A({\mathfrak{h}})].$$
   Therefore $[\mathfrak{h}_1,\widetilde{C}_A({\mathfrak{h}})]$ is contained in $\widetilde{C}_A(\mathfrak{g})$, which is trivial, and this shows that $\mathfrak{g}/\mathfrak{h}_1$ acts on $\widetilde{C}_A(\mathfrak{h})$.\\
   To prove the lemma, it suffices to show that $inf'$ is injective and $\mathrm{im}(inf')=\ker(res)$.
   \begin{itemize}
       \item We first verify that $\ker(inf')$ is trivial. Let $f\in \widetilde{H}^1(\mathfrak{g}/\mathfrak{h}_1,\widetilde{C}_A(\mathfrak{h}))$ be such that $inf'(f)\sim da$ for some $a\in A$. Then $inf(f)(\mathfrak{h}_1)=\{0\}$ by the definition of $inf$. Therefore $f(\mathfrak{h}_1)\sim da(\mathfrak{h}_1)=\mathfrak{h}_1(a)$ is finite, and so $a\in \widetilde{C}_A(\mathfrak{h}_1)$. Since $\mathfrak{h}_1$ has finite index in $\mathfrak{h}$ by Lemma \ref{symact}, we may conclude that $a\in \widetilde{C}_A(\mathfrak{h})$. Consequently, the element $[f]$ equals $[da]$ in $\widetilde{H}^1(\mathfrak{g}/\mathfrak{h}_1,\widetilde{C}_A(\mathfrak{h}))$.
       \item We show that $\mathrm{im}(inf')\subseteq \ker(res)$. Let $[s]\in \widetilde{H}^1(\mathfrak{g},A)$ and assume that $[s]=[inf(f)]$ for some $f\in \widetilde{H}^1(\mathfrak{g}/\mathfrak{h}_1,\widetilde{C}_A(\mathfrak{h}))$. Then $s\sim inf(f)+da$ for some $a\in A$. Applying $res$, we obtain that $res(inf'(s))\sim 0$ and $res(da)\sim da$, therefore $res[s]=[0]$.
       \item We prove that $\mathrm{im}(inf')\supseteq \ker(res)$. Let $[s]\in \widetilde{H}^1(\mathfrak{g},A)$ be such that $res[s]=[0]$ in $\widetilde{H}^1(\mathfrak{h},A)$. Then there exists some $a\in A$ such that $res(s)\sim da$ when restricted to $\mathfrak{h}$. After replacing $s$ by $s-da\in \widetilde{C}^1(\mathfrak{g},A)$, we may assume that $s(\mathfrak{h}\cap \Dom(s))$ is finite. We may restrict the domain of $s$ to the subgroup $D^s\leq \mathfrak{g}$ introduced in Definition \ref{DefAlmDer}. We first prove that $s(\Dom(s))\leq \widetilde{C}_{A}(\mathfrak{h})$. By assumption, for every $g\in \Dom(s)$, the subgroup $D^s_g$ (Definition \ref{DefAlmDer}) has finite index in $\mathfrak{g}$, and 
   $$(\mathfrak{h}\cap D^s_g)s(g)\leq s([g,\mathfrak{h}\cap D^s_g])+gs(\mathfrak{h}\cap D^s_g).$$
   Since $\mathfrak{h}$ is an ideal in $\mathfrak{g}$, we have that $s([g,D^s_g\cap \mathfrak{h}])\leq s(\mathfrak{h})$ and $g(s(\mathfrak{h}\cap D^s_g))$ is finite. Consequently, the subgroup $\mathfrak{h}(s(g))$ is finite for every $g\in \Dom(s)$. This shows that $\mathrm{im}(s)\leq \widetilde{C}_{A}(\mathfrak{h})$.\\
   We prove that $s(\mathfrak{h}_1)$ is a finite $\mathfrak{g}_1$-module for some definable Lie subring $\mathfrak{g}_1\sqsubseteq \mathfrak{g}$ of finite index in $\mathfrak{g}$. As $s(\mathfrak{h}_1)$ is finite, the $\mathfrak{g}_1$-invariance implies that $s(\mathfrak{h}_1)\leq \widetilde{C}_A(\mathfrak{g})=\{0\}$. As $s(\mathfrak{h}\cap \Dom(s))$ is finite and $\mathfrak{h}_1$ is a Lie subring of $\mathfrak{h}$, we have that $s(\mathfrak{h}_1)=\{s(h_i)\}_{i\leq n}$ for some $n\in \mathbb{N}$. Let $D^s_{\mathfrak{h}_1}:=\bigcap_{i\leq n} D^s_{h_i}$. Since each $D^s_{h_i}$ has finite index in $\mathfrak{g}$ by hypothesis, also $D^s_{\mathfrak{h}_1}\leq \mathfrak{g}$ has finite index in $\mathfrak{g}$. Then, for every $g\in D^s_{\mathfrak{h}_1}$,
   $$g\big(s(\mathfrak{h}_1\cap \Dom(s))\big)\leq s([\mathfrak{h}_1,g]\cap \Dom(s))+\mathfrak{h}_1s(g).$$
   The term $s([\mathfrak{h}_1,g]\cap \Dom(s))$ is contained in $s(\mathfrak{h}_1)$ since $\mathfrak{h}_1$ is an ideal of $\mathfrak{g}$. The term $\mathfrak{h}_1s(g)$ is contained in $[\mathfrak{h}_1,\widetilde{C}_A(\mathfrak{h})]$ since $\mathrm{im}(s)\subseteq \widetilde{C}_A(\mathfrak{h})$. The subgroup $[\mathfrak{h}_1,\widetilde{C}_A(\mathfrak{h})]$ is finite by Lemma \ref{fin[]}, and it is $\mathfrak{g}$-invariant since $\widetilde{C}_A(\mathfrak{h})$ is $\mathfrak{g}$-invariant and $\mathfrak{h}_1$ is an ideal in $\mathfrak{g}$. Therefore, the subgroup $s(\mathfrak{h}_1)$ is contained in $\widetilde{C}_A(\mathfrak{g})=\{0\}$, and hence $s$ is an almost derivation. Consequently $[s]\in \widetilde{H}^1(\mathfrak{g}/\mathfrak{h}_1,\widetilde{C}_A(\mathfrak{h}))$. This completes the proof.
   \end{itemize}
   \end{proof}
    The following result proves the map $res$ is injective when $\mathfrak{h}$ has finite index in $\mathfrak{g}$.
   \begin{lemma}\label{H1finind}
   Let $\mathfrak{g}$ be a Lie ring acting on the abelian group $A$, both definable in a finite-dimensional theory. Let $\mathfrak{h}\sqsubseteq \mathfrak{g}$ be a definable Lie subring of finite index in $\mathfrak{g}$. Then $\widetilde{H}^1(\mathfrak{g},A)$ embeds into $\widetilde{H}^1(\mathfrak{h},A)$.
   \end{lemma}
   \begin{proof}
   Let $res$ be the map
   $$res:\widetilde{H}^1(\mathfrak{g},A)\to \widetilde{H}^1(\mathfrak{h},A)$$
   as previously defined. We verify that $res$ is injective. Let $[f]\in \widetilde{H}^1(\mathfrak{g},A)$ be such that $f_{|\mathfrak{h}}\sim da$ for some $a\in A$. Then $(f-da)(\mathfrak{h})$ is finite. The same holds for $(f-da)(\mathfrak{g})$, and so $[f]=[0]$ in $\widetilde{H}^1(\mathfrak{g},A)$. 
   \end{proof}
   The following lemma shows that the image of $res$ is contained in $\widetilde{C}_{\mathfrak{g}}(\widetilde{H}^1(\mathfrak{h},A))$.
   \begin{lemma}\label{ImRes}
   Let $\mathfrak{g}$ be a Lie ring acting on an abelian group $A$, both definable in a finite-dimensional theory. Let $\mathfrak{h}$ be a definable Lie subring of $\mathfrak{g}$. Then $res(\widetilde{H}^1(\mathfrak{g},A))\leq \widetilde{C}_{\mathfrak{g}}(\widetilde{H}^1(\mathfrak{h},A))$.
   \end{lemma}
   \begin{proof}
   Let $[f]\in \widetilde{H}^1(\mathfrak{h},A)$ be such that $f\sim s_{|\mathfrak{h}}$ for some $s\in \widetilde{H}^1(\mathfrak{g},A)$. Then, by the same argument of Lemma \ref{actalmcoh}, we obtain that $g\cdot s\sim d(s(g))$ for every $g\in D^f$. Hence, for every $h\in D^f_g\cap \ker(f-s_{|\mathfrak{h}})$, 
   $$(g\cdot f)(h)=f[g,h]+gf(h)=s[g,h]+gs(h)=hs(g)$$
   and so $g\cdot f\in [\operatorname{InnDer}(\mathfrak{h},A)]$. The lemma follows from the arbitrariness of $[f]\in \widetilde{H}^1(\mathfrak{h},A)$.
   \end{proof}
   As a corollary of Lemma \ref{ImRes}, we show that, if we quotient $\mathfrak{g}$ by an almost central ideal contained in $\widetilde{C}_{\mathfrak{g}}(A)$, the $1$st almost cohomology group that we obtain is isogenous to the original one.
   \begin{corollary}\label{QUoAlmCen}
   Let $(\mathfrak{g},A)$ be a definable $\mathfrak{g}$-module of finite dimension with $\widetilde{C}_A(\mathfrak{g})=\{0\}$. Let $\mathfrak{h}\mathrel{\unlhd}\mathfrak{g}$ be a definable ideal such that $\mathfrak{h}\leq \widetilde{C}_{\mathfrak{g}}(A)\cap \widetilde{Z}(\mathfrak{g})$. Then $\widetilde{H}^1(\mathfrak{g}/\mathfrak{h},A)$ is isogenous to $\widetilde{H}^1(\mathfrak{g}/A)$.
   \end{corollary}
   \begin{proof}
   Since $\mathfrak{h}$ is contained in $\widetilde{C}_{\mathfrak{g}}(A)$, the subgroup $\widetilde{C}_{A}(\mathfrak{h})$ has finite index in $A$ by Lemma \ref{symact}. By Lemma \ref{A/A_1}, the $0$th and the $1$st almost cohomology group of $(\mathfrak{g},A)$ are isogenous respectively to the $0$th and the $1$st almost cohomology group of $(\mathfrak{g},\widetilde{C}_{\mathfrak{h}}(A))$. Hence, without loss of generality, we may assume $A=\widetilde{C}_A(\mathfrak{h})$. By Lemma \ref{exact series}, the sequence
   $$0\xrightarrow{} \widetilde{H}^1(\mathfrak{g}/\mathfrak{h}_1,A)\xrightarrow{inf} \widetilde{H}^1(\mathfrak{g},A)\xrightarrow{res} \widetilde{C}_{\mathfrak{g}}(\widetilde{H}^1(\mathfrak{h},A))$$
   is exact, where $\mathfrak{h}_1=\widetilde{C}_{\mathfrak{h}}(A)=\mathfrak{h}$. We show that $\widetilde{C}_{\mathfrak{g}}(\widetilde{H}^1(\mathfrak{h},A))$ is trivial. Fix $f\in \widetilde{C}_{\mathfrak{g}}(\widetilde{H}^1(\mathfrak{h},A))$. By definition, there exists a definable subgroup $G_1\leq \mathfrak{g}$ of finite index such that, for every $g\in G_1$, the partial homomorphism $g\cdot f$ is equivalent to $d(a_g)$ for some $a_g\in A$. By Lemma \ref{symact}, the ideal $\widetilde{C}_{\mathfrak{g}}(\mathfrak{h})\mathrel{\unlhd} \mathfrak{g}$ has finite index in $\mathfrak{g}$. We verify that $f( C_{\mathfrak{h}}(a_g) \cap C_{\mathfrak{h}}(g))=0$ for every $g\in G_1\cap \widetilde{C}_{\mathfrak{g}}(\mathfrak{h})$. Let $h\in C_{\mathfrak{h}}(a_g) \cap C_{\mathfrak{h}}(g)$, then 
   $$0=h(a_g)=d(a_g)(h)=(g\cdot f)(h)=gf(h)-f([g,h])=gf(h)$$
   since $h\in C_{\mathfrak{g}}(g)$. As $C_{\mathfrak{h}}(a_g) \cap C_{\mathfrak{h}}(g)\leq \mathfrak{h}$ has finite index in $\mathfrak{h}$ by hypothesis, the subgroup $gf(\mathfrak{h})$ is finite for every $g\in G_1\cap \widetilde{C}_{\mathfrak{g}}(\mathfrak{h})$. Therefore $\widetilde{C}_{\mathfrak{g}}(f(\mathfrak{h}))$ has finite index in $\mathfrak{g}$. By Lemma \ref{symact}, we have that $f(\mathfrak{h})\apprle \widetilde{C}_A(\mathfrak{g})=\{0\}$. Therefore $f(\mathfrak{h})$ is finite, and so $f\sim 0$. This completes the proof. 
   \end{proof}
    \section{Finiteness of $\widetilde{H}^1(\mathfrak{g},A)$}
    We now prove Theorem \ref{TrivialityH1}, which we recall.
    \begin{theorem}
    Let $\mathfrak{g}$ be a nilpotent Lie ring acting on an abelian group $A$, both definable in a finite-dimensional theory. Assume that $\widetilde{C}_A(\mathfrak{g})$ is finite. Then $\widetilde{H}^1(\mathfrak{g},A)$ is finite.
    \end{theorem}
    \begin{proof}
    We may assume that $\widetilde{C}_A(\mathfrak{g})$ is trivial. Indeed, by Lemma \ref{A/A_2}, the group $\widetilde{H}^1(\mathfrak{g},A)$ is isogenous to $\widetilde{H}^1(\mathfrak{g},A/\widetilde{C}_{A}(\mathfrak{g}))$, and so, if the latter is finite, also $\widetilde{H}^1(\mathfrak{g},A)$ is finite. Working with the module $(\mathfrak{g},A/\widetilde{C}_A(\mathfrak{g}))$, we may assume that $\widetilde{C}_{A}(\mathfrak{g})$ is trivial by Lemma \ref{g/Z(g)}.\\
    We argue by induction on the dimension of $A$. Let $\mathcal{C}$ be the collection of all series $(A_i)_{i\leq m}$ of definable $\mathfrak{h}$-submodules of $A$, where $\mathfrak{h}$ varies in the definable Lie subrings of finite index in $\mathfrak{g}$, such that every factor $A_i/A_{i+1}$ is an infinite $\mathfrak{h}$-minimal module. Let $\widetilde{lg}(A)$ denote the maximal length of a series in $\mathcal{C}$. Let $(A_i)_{i\leq \widetilde{lg}(A)}\in \mathcal{C}$ a series of maximal length, and let $\mathfrak{h}$ be a definable Lie subring of finite index in $\mathfrak{g}$ such that every $A_i$ is $\mathfrak{h}$-invariant.\\
    Assume first that $\widetilde{lg}(A)=1$. Then $A$ is strongly absolutely $\mathfrak{g}$-minimal by Lemma \ref{LinAlmAbe2}. Hence, by Theorem \ref{LinAlmNil}, there exist a definable Lie subring $\mathfrak{h}$ of finite index in $\mathfrak{g}$ and a $\mathfrak{h}$-submodule $A_1\leq A$ of finite index in $A$ such that $A_1\simeq K^+$ for some definable field $K$ and such that $\mathfrak{h}/\widetilde{C}_{\mathfrak{h}}(A_1)$ definably embeds into $K^+\cdot \mathrm{Id}$. By Lemma \ref{H1finind}, it suffices to prove that $\widetilde{H}^1(\mathfrak{h},A)$ is finite. Hence, we may assume that $\mathfrak{g}=\mathfrak{h}$. Moreover, applying iteratively Lemma \ref{QUoAlmCen}, we may assume that $\widetilde{C}_{\mathfrak{g}}(A)$ is trivial. Similarly, applying Lemma \ref{A/A_2}, we may assume that $A_1=A$. Let 
    $$p:\mathfrak{g}\to K^+$$
    be the embedding given by Theorem \ref{LinAlmNil}. Then, for every almost derivation $\sigma$, choose $y\in D^{\sigma}\setminus\{0\}$. Since $p(y)$ is invertible in $K$, we can define the element $a=:p(y)^{-1}\sigma(y)\in K^+\simeq A$. Since $\mathfrak{g}$ is abelian, we have that 
    $$p(h)\cdot \sigma(y)=\sigma([h,y])+p(y)\cdot\sigma(h)=p(y)\cdot\sigma(h)$$ for every $h\in D^{\sigma}_y$. Hence, for every $h\in D^{\sigma}_y$,
   $$p(h)\cdot a=p(h)\cdot p(y)^{-1}\sigma(y)=p(y)^{-1}\cdot p(h)\cdot \sigma(y)=p(y)^{-1}p(y)\sigma(h)=\sigma(h).$$
    Thus, the partial homomorphism $\sigma$ is equivalent to $d(a)$, and so $\sigma\in [\operatorname{InnDer}(\mathfrak{g},A)]$. In conclusion, $\widetilde{H}^1(\mathfrak{g},A)$ is trivial, and this proves the claim.\\ 
   We now continue with the inductive step. Assume that $\widetilde{lg}(A)=n$. Let $B$ be a definable infinite $\mathfrak{h}$-submodule of infinite index in $A$, where $\mathfrak{h}\sqsubseteq \mathfrak{g}$ is a Lie subring of finite index in $\mathfrak{g}$. Clearly, $\widetilde{C}_A(\mathfrak{h})=\widetilde{C}_A(\mathfrak{g})$, which is trivial. By the induction hypothesis, $\widetilde{H}^1(\mathfrak{h},B)$ is finite. Let 
   $$0\to B\xrightarrow{i} A\xrightarrow{p}A/B\to 0$$
   be an exact sequence of $\mathfrak{h}$-modules. By Lemma \ref{longseries}, the sequence
   $$0\to \widetilde{H}^0(\mathfrak{h},B)\xrightarrow{i} \widetilde{H}^0(\mathfrak{h},A)\xrightarrow{p} \widetilde{H}^0(\mathfrak{h},A/B)\xrightarrow{\delta_1} \widetilde{H}^1(\mathfrak{h},B)\xrightarrow{i^1} \widetilde{H}^1(\mathfrak{h},A)\xrightarrow{p^1} \widetilde{H}^1(\mathfrak{h},A/B)$$
   is exact. Since $\widetilde{H}^0(\mathfrak{h},A)$ is trivial and $\widetilde{H}^1(\mathfrak{h},B)$ is finite, it follows that $\widetilde{H}^0(\mathfrak{h},A/B)$ is finite. By the induction hypothesis, the group $\widetilde{H}^1(\mathfrak{h},A/B)$ is finite. By the previous exact sequence, we conclude that $\widetilde{H}^1(\mathfrak{h},A)$ is finite. The theorem then follows from Lemma \ref{H1finind}. 
    \end{proof}
We now prove some fundamental corollaries of Theorem \ref{TrivialityH1}.
    \begin{corollary}\label{FinH0UV}
    Let $\mathfrak{g}$ be a nilpotent Lie ring and $A$ a $\mathfrak{g}$-module, both definable in a finite-dimensional theory, and let $A\geq U\geq V$ be two definable $\mathfrak{g}$-submodules of $A$. Assume that $\widetilde{C}_A(\mathfrak{g})$ is finite. Then $\widetilde{C}_{\mathfrak{g}}(U/V)$ is finite. 
     \end{corollary}
    \begin{proof}
    Since $\widetilde{C}_A(\mathfrak{g})$ is finite, the same holds for $\widetilde{C}_U(\mathfrak{g})$ and $\widetilde{C}_V(\mathfrak{g})$. By Theorem \ref{TrivialityH1}, the groups $\widetilde{H}^1(\mathfrak{g},U)$ and $\widetilde{H}^1(\mathfrak{g},V)$ are both finite. Since the sequence 
   $$0\to V\xrightarrow{i} U\xrightarrow{p} U/V\to 0$$ 
   is exact, then the sequence
   $$0\mapsto \widetilde{H}^0(\mathfrak{g},V)\xrightarrow{i} \widetilde{H}^0(\mathfrak{g},U)\xrightarrow{p} \widetilde{H}^0(\mathfrak{g},U/V)\xrightarrow{\delta_1} \widetilde{H}^1(\mathfrak{g},V)\xrightarrow{i^1} \widetilde{H}^1(\mathfrak{g},U)\xrightarrow{p^1} \widetilde{H}^1(\mathfrak{g}, U/V)$$
   is exact by Lemma \ref{longseries}. As $\widetilde{H}^0(\mathfrak{g},U)$ and $\widetilde{H}^1(\mathfrak{g},V)$ are both finite, also $\widetilde{H}^0(\mathfrak{g},U/V)$ is finite.
    \end{proof}
    Corollary \ref{FinH0UV} allows us to characterise the action of a nilpotent Lie ring.
    \begin{theorem}
        Let $\mathfrak{g}$ be a nilpotent Lie ring acting on a module $A$, both definable in a finite-dimensional theory. Then there exist
        \begin{itemize}
            \item A natural $n\in \mathbb{N}$ such that $\widetilde{C}^n_A(\mathfrak{g})=\widetilde{C}^{n+1}_A(\mathfrak{g})$;
            \item A definable Lie subring $\mathfrak{h}\sqsubseteq \mathfrak{g}$ of finite index;
            \item A chain $(A_i)_{i=0}^n$ of $\mathfrak{h}$-modules such that
            \begin{itemize}
                \item $A_0=\widetilde{C}^n_A(\mathfrak{g})$;
                \item $A_n\leq A$ has finite index in $A$; 
                \item $A_i\geq A_{i-1}$ for every $i\leq n$;
                \item $A_i/A_{i-1}$ is isogenous to the additive group of an infinite definable field $K_i$ of finite dimension.
            \end{itemize}
        \end{itemize} 
    \end{theorem}
    \begin{proof}
        Let $\widetilde{C}^{n-1}_{A}(\mathfrak{g})$ be an iterated almost centraliser of maximal possible dimension. Then, by the same proof of Lemma \ref{g/Z(g)}, we have that $\widetilde{C}^n_A({\mathfrak{g}})=\widetilde{C}^{n+1}_A({\mathfrak{g}})$. Hence, the group $\widetilde{C}_{A/\widetilde{C}^n_{A}(\mathfrak{g})}(\mathfrak{g})$ is trivial. Let $(A_i)_{i\leq n}$ be a series of $\mathfrak{h}$-modules, where $\mathfrak{h}\sqsubseteq \mathfrak{g}$ is a definable Lie subring of finite index, such that $A_{i+1}/A_i$ is strongly absolutely $\mathfrak{h}$-minimal. This series exists by Lemma \ref{ObsStrAbsMin}. Moreover, by Corollary \ref{FinH0UV}, the group $\widetilde{C}_{\mathfrak{h}}(A_i/A_{i-1})$ is finite for every $i\leq n$. Then the conclusion follows from Theorem \ref{LinAlmNil}.
    \end{proof}
    Another application of Theorem \ref{TrivialityH1} concerns almost Cartan subrings. 
    \begin{corollary}
    Let $\mathfrak{g}$ be a Lie ring definable in a finite-dimensional theory. Let $\mathfrak{i}$ be a definable ideal of $\mathfrak{g}$ and $\mathfrak{c}$ a definable almost Cartan subring of $\mathfrak{i}$. Then $\mathfrak{g}\sim \mathfrak{i}+\widetilde{N}_{\mathfrak{g}}(\mathfrak{c})$.
    \end{corollary}
    \begin{proof}
    Up to passing to a Lie subring of finite index, we may assume that $\mathfrak{c}$ is nilpotent.\\
    The action of $\mathfrak{c}$ on $(\mathfrak{g}/\mathfrak{n},+)$, where $\mathfrak{n}:=\widetilde{N}_{\mathfrak{g}}(\mathfrak{c})$, is well-defined since $[\mathfrak{c},\mathfrak{n}]\leq \mathfrak{n}$. Indeed, for every $n\in \mathfrak{n}$ and $c\in \mathfrak{c}$, 
    $$[[n,c],\mathfrak{c}]\leq [n,[c,\mathfrak{c}]]+[c,[n,\mathfrak{c}]]\apprle [n,\mathfrak{c}]+[c,\mathfrak{c}]\apprle \mathfrak{c}.$$
    We show that $\widetilde{C}_{(\mathfrak{g}/\mathfrak{n},+)}(\mathfrak{c})$ is trivial. Let $g\in \widetilde{C}_{(\mathfrak{g}/\mathfrak{n},+)}(\mathfrak{c})$. Then $[g,\mathfrak{c}]\apprle \mathfrak{n}$ and, since $\mathfrak{i}$ is an ideal of $\mathfrak{g}$, we obtain that $[g,\mathfrak{c}]\apprle \mathfrak{n}\cap \mathfrak{i}\apprle \mathfrak{c}$. Consequently $g\in \mathfrak{n}$ and, by arbitrariness of $g\in \widetilde{C}_{(\mathfrak{g}/\mathfrak{n})}(\mathfrak{c})$, we conclude that $\widetilde{C}_{(\mathfrak{g}/\mathfrak{n})}(\mathfrak{c})$ is trivial. Therefore, by Corollary \ref{FinH0UV}, the group $\widetilde{C}_{\mathfrak{g}/\mathfrak{i}+\mathfrak{n}}(\mathfrak{c})$ is finite. Conversely $[\mathfrak{c},\mathfrak{g}]\leq \mathfrak{i}$ since $\mathfrak{i}$ is an ideal. This yields a contradiction unless $\mathfrak{i}+\mathfrak{n}$ has finite index in $\mathfrak{g}$.
    \end{proof}

\section{Conclusion}\label{QuestionLieRing}
This article should be seen as a starting point for the application of cohomology tools in the analysis of finite-dimensional Lie rings. \\
The natural application for this article is to the classification of almost Cartan Lie rings. In the finite Morley rank framework, several fundamental results have been obtained by Zamour and Tindzogho Ntsiri in \cite{zamourCar}. An extension of \cite{zamourCar} to the finite-dimensional setting is essential for an analysis of NIP (absolutely) anisotropic Lie rings (see \cite{invitti2025lie}), which always contains an almost Cartan Lie subring by Lemma \cite[Lemma 7.8]{invitti2025lie}.
\begin{question}
    Analyse almost Cartan Lie subrings.
\end{question}
Another line of research is the analysis of cohomology groups for soluble non-nilpotent Lie rings, aiming for results similar to those of \cite{barnes1967cohomology}.
\begin{question}
    Determine the properties of $\widetilde{H}^1(\mathfrak{g},A)$ for a soluble definable Lie ring $\mathfrak{g}$ of finite dimension.
\end{question}
The main obstacle to further development of the almost cohomology theory for Lie rings lies in the construction of the $i$th almost cohomology groups, with $i\geq 2$. In particular, the second almost cohomology group may have an important role in the study of Lie rings of finite dimension. Indeed, Lie algebras with trivial second cohomology group for the adjoint action, called \emph{cohomologically rigid}, have been studied by several authors \cite{jose1992rang,goze2001classification}. The main obstacle is that, if we try to mimic the construction in \cite{zamourCohom} for $d(f)$ where $f\in \widetilde{C}^1(\mathfrak{g},A)$, the domain of $d(f)$ is not necessarily of finite index. This problem does not occur in the NIP finite-dimensional case, as a consequence of \cite[Lemma 2.10]{invitti2025lie}. This observation highlights the tame behaviour of NIP Lie rings of finite dimension, which is already suggested by \cite{invitti2025lie}.

\end{document}